\documentclass[authoryear,preprint,11pt]{elsarticle}

\usepackage[margin=1.7cm]{geometry}

\usepackage{amssymb}
\usepackage{bm}
\usepackage{amsthm}
\usepackage{lipsum}
\usepackage{comment}
\usepackage[fleqn]{amsmath}
\newtheorem{theorem}{Theorem}[section]

\newtheorem{assumption}{Assumption}
\usepackage{algorithm}
\usepackage{algpseudocode}
\usepackage{tabularx}

\DeclareMathOperator*{\argmin}{arg\,min}
\usepackage{threeparttable}
\usepackage{booktabs} 
\usepackage{multirow}
\usepackage{tabularx}
\usepackage{url} 
    
\usepackage{hyperref}
\allowdisplaybreaks
\usepackage{xcolor}
\usepackage{longtable}
\algnewcommand\algorithmicinput{\textbf{Input:}}
\algnewcommand\INPUT{\item[\algorithmicinput]}
\algdef{SE}[IF]{If}{EndIf}[1]
  {\algorithmicif\ #1\ \algorithmicthen}{\algorithmicend\ \algorithmicif}

\journal{European Journal of Operational Research}

\begin{document}

\begin{frontmatter}



\title{Scalable Multi-Level Optimization for Sequentially Cleared Energy Markets with a Case Study on Gas and Carbon Aware Unit Commitment}




\author[inst1]{Yuxin Xia \texorpdfstring{\corref{cor1}}{}}
\cortext[cor1]{Corresponding author}
\affiliation[inst1]{organization={School of Engineering, The University of Edinburgh},
            city={Edinburgh},
            postcode={EH9 3FB}, 
            country={UK}}

\author[inst2]{Iacopo Savelli}
\author[inst3]{Thomas Morstyn}
\affiliation[inst2]{organization={GREEN center, Bocconi University},
            city={Milano},
            postcode={20136}, 
            country={Italy}}
\affiliation[inst3]{organization={Department of Engineering Science, University of Oxford},
            city={Oxford},
            postcode={OX3 7DQ}, 
            country={UK}}
\begin{abstract}
This paper examines Mixed-Integer Multi-Level problems with Sequential Followers (MIMLSF), a specialized optimization model aimed at enhancing upper-level decision-making by incorporating anticipated outcomes from lower-level sequential market-clearing processes. We introduce a novel approach that combines lexicographic optimization with a weighted-sum method to asymptotically approximate the MIMLSF as a single-level problem, capable of managing multi-level problems exceeding three levels. To enhance computational efficiency and scalability, we propose a dedicated Benders decomposition method with multi-level subproblem separability. To demonstrate the practical application of our MIMLSF solution technique, we tackle a unit commitment problem (UC) within an integrated electricity, gas, and carbon market clearing framework in the Northeastern United States, enabling the incorporation of anticipated costs and revenues from gas and carbon markets into UC decisions. This ensures that only profitable gas-fired power plants (GFPPs) are committed, allowing system operators to make informed decisions that prevent GFPP economic losses and reduce total operational costs under stressed electricity and gas systems. The case study not only demonstrates the applicability of the MIMLSF model but also highlights the computational benefits of the dedicated Benders decomposition technique, achieving average reductions of 32.23\% in computing time and 94.23\% in optimality gaps compared to state-of-the-art methods.
\end{abstract}




\end{frontmatter}


\section{Introduction}
Sequential decision-making is fundamental to the clearing processes of diverse energy markets, which are influenced by temporal, spatial, operational, and hierarchical dimensions. Moreover, critical operations such as unit commitment (UC) are executed ahead of time, often without full knowledge of subsequent sequential market outcomes. This lack of foresight and awareness can result in suboptimal or economically inefficient decisions, as the initial decisions may fail to incorporate the future market conditions that subsequently unfold. To this end, a hierarchical approach is needed, enabling decision-makers such as system operators, regulators, and strategic agents to anticipate and adjust for future market conditions, thereby improving risk management and decision-making.

\begin{figure}[htp]
    \centering
    \includegraphics[width=0.7\columnwidth]{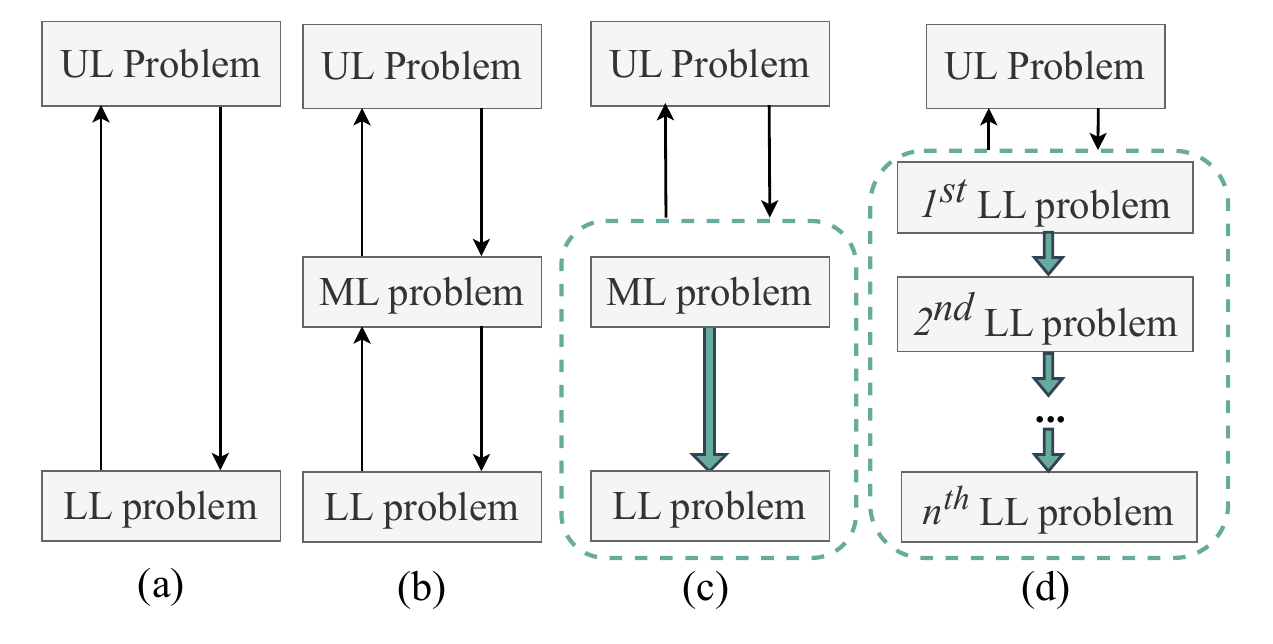}
    \caption{Different models of hierarchical decision making process. UL: Upper-Level, ML: Middle-Level, LL: Lower-Level.  (a) classic bi-level problem; (b) classic tri-level problem; (c) tri-level problem with sequential ML and LL problems; (d) multi-level problem with $n$ sequential LL problems.}
    \label{fig:framework}
\end{figure}
To mathematically formulate this hierarchical decision-making framework, multi-level optimization problems have been extensively studied. In classic bi-level programming, the structure comprises an upper-level leader and a lower-level follower, with the follower’s actions acting as constraints on the leader’s decisions, as represented in Fig.~\ref{fig:framework}(a). This bi-level optimization programming is extensively reviewed in the literature, see \citet{BECK2023401} and \citet{KLEINERT2021100007}. Multi-level programming, with more than two levels, has been gaining increasing attention due to its applicability to real-world problems \citep{LU2016463}. Multi-level programming can address nested hierarchies where multiple agents make decisions in sequence driven by individual objectives. As an example, the structure of a classic tri-level problem with a nested decision-making hierarchy is depicted in Fig.~\ref{fig:framework}(b) \citep{avraamidou2019}. In this classic tri-level optimization problem, the first decision-maker at the upper-level formulates an optimization problem whose constraints include the optimization problem of the second decision-maker. This second decision-maker's problem further incorporates the optimization problem of a third decision-maker in the constraints, making for a nested decision-making hierarchy \citep{avraamidou2019}. Recently, a special tri-level structure involving two sequential followers has been studied in \citet{8844828}, as shown in Fig.~\ref{fig:framework}(c). In this special framework, unlike the classic tri-level problem, the middle-level problem does not anticipate the solutions of the lower-level problem. This special tri-level optimization frameworks have been demonstrated in various power system applications. These include heat unit commitment problems with sequential heat-electricity markets \citep{mitridati2019bidvalidity}. Additionally, for applications involving sequential national and local market operations, this framework has been applied to both optimal bidding strategies for flexible service providers \citep{10151915} and merchant transmission investment \citep{XIA2025124721}.

Multi-level problems, even for the case of bi-levels, are recognized as strongly $\mathcal{NP}$-hard, presenting significant computational challenges \citep{FAKHRY20221114}. Unlike bi-level problems, tri-level and more complex multi-level problems often cannot be efficiently solved using traditional methods such as Karush–Kuhn–Tucker (KKT) conditions \citep{avraamidou2019}. To address complex multi-level problems, researchers have developed both decomposition methods and column-and-constraint generation approaches \citep{8636266,9521770,Veronika2019,7748594,7930419,FLORENSA2017449,9068456}. Additionally, multi-parametric and heuristic approaches offer alternative solutions \citep{avraamidou2019,ELMELIGY2021107274,FAKHRY20221114,CHOUHAN2022109468}. Note that for the specific tri-level problem with two sequential problems as shown in Fig.~\ref{fig:framework}(c), dedicated methods such as lexicographic asymptotic approximation and a specialized Benders decomposition have been proposed in \citet{8844828, doi:10.1287/ijoc.2021.1128}. To the best knowledge of the authors, there is a lack of efficient methods to solve multi-level problems involving sequential market operations. Additionally, the use of lexicographic and weighted-sum asymptotic approximation techniques introduces numerical scalability issues in multi-level optimization problems with multiple lower-level agents, as the weighted-sum scaling terms progressively approach zero with each subsequent lower level, leading to potential computational instability. Furthermore, classic Benders decomposition approaches face substantial computational challenges, particularly in handling complex mixed-integer problems. These challenges arise because Benders subproblems contain a large number of variables and constraints, making them computationally intensive to solve. 


Sequential decision-making is fundamental to the energy sector, encompassing temporal, spatial, operational, and hierarchical dimensions of market operations. For example, electricity markets are organized hierarchically and cleared sequentially, each operating at different \emph{temporal} resolutions ranging from months to minutes. This structure has been widely modeled in strategic bidding \citep{WOZABAL2020639, RINTAMAKI20201136, QIN20231227} and energy management \citep{9758051, 9492111}. Moreover, while markets for different energy sectors are \emph{operationally} independent, they are becoming increasingly interconnected. For example, although the electricity and natural gas markets typically operate independently, they are interconnected through gas-fired power plants (GFPPs). Currently, these two markets are cleared sequentially in the United States, and there is significant research interest in how their interdependence affects problems including generation expansion planning \citep{8395045}, strategic bidding \citep{DIMITRIADIS2023127710} and UC \citep{8844828}. Additionally, the sequence in which electricity and heat markets are cleared \citep{mitridati2019bidvalidity, MITRIDATI20201107}, as well as the design of carbon markets--which can be cleared either simultaneously or sequentially with the electricity market \citep{DIMITRIADIS2023127710, 9913314, 9380720}--will affect market outcomes. The \emph{temporal and operational} dimensions are also crucial for current market models, particularly impacting the coordination of day-ahead and real-time markets for electricity and natural gas. A range of coordination approaches have been proposed, ranging from sequentially decoupled to fully uncoordinated and purely stochastic approaches \citep{ORDOUDIS2019642, schwele2021coordination, 9492841}. In addition, recent research has increasingly focused on the \emph{spatial and temporal} dynamics of local electricity markets. Notably, \citet{KHAN2022107624} have developed a bidding strategy for producers participating in sequential market clearing at wholesale and retail levels. Moreover, the study in \citet{10151915} explores a sequential market clearing framework, which models the stacking of flexibility revenues of distributed energy resources in sequential national and local markets.

As previously discussed, GFPPs play a critical role in linking the electricity and gas markets. These plants are notable for providing environmental benefits, low capital costs, high efficiency, and operational flexibility \citep{7913721}. Nevertheless, GFPPs encounter significant challenges when these two markets operate independently \citep{8844828}. This is particularly problematic when GFPPs must bid in the electricity market without knowledge of actual gas prices (i.e., fuel costs), which can lead to substantial economic losses, as evidenced during the 2014 polar vortex \citep{7029695, 8844828, PJM2014ColdWeather} in the Northeastern United
States. Proposals to enhance coordination between electricity and gas networks have been suggested \citep{ORDOUDIS2019642, schwele2021coordination, 9492841}, yet these markets typically clear independently, heightening financial risks during high stresses for GFPPs on both networks. In response, a new bid-validity constraint has been introduced in \citet{8844828} to ensure that only profitable GFPPs can be committed, taking into account solely the gas costs (i.e., fuel costs).

The carbon cap-and-trade system, which enables the trading of emission allowances within regulatory constraints \citep{EU_ETS}, is a cost-effective mechanism for reducing greenhouse gas emissions and promoting decarbonization across the energy sector. Recent studies have focused on integrating carbon markets with traditional energy sectors, including electricity \citep{9999383,10330727,10530198,LI2024122328}, natural gas, and heat markets \citep{DIMITRIADIS2023127710,yang2022identifying,10293189}, to optimize environmental and economic outcomes. The introduction of the carbon market has redefined the roles of GFPPs, which traditionally participate as sellers in the electricity market and buyers in the gas market. Now, these entities also engage as either sellers or buyers in the carbon market, depending on their emission levels and allowances. Notably, the study by \citet{9380720} provides an analysis of GFPPs' participation in the electricity, gas, and carbon markets, focusing on market equilibria and the potential for market power exploitation by strategic energy producers. Despite advancements in modeling GFPPs' participation, there exists a critical research gap in the need for a UC model that integrates both carbon and gas economic feedback to prevent GFPP defaults and reduce total operational costs.

\begin{figure}[tp]
    \centering
    \includegraphics[width=1\columnwidth]{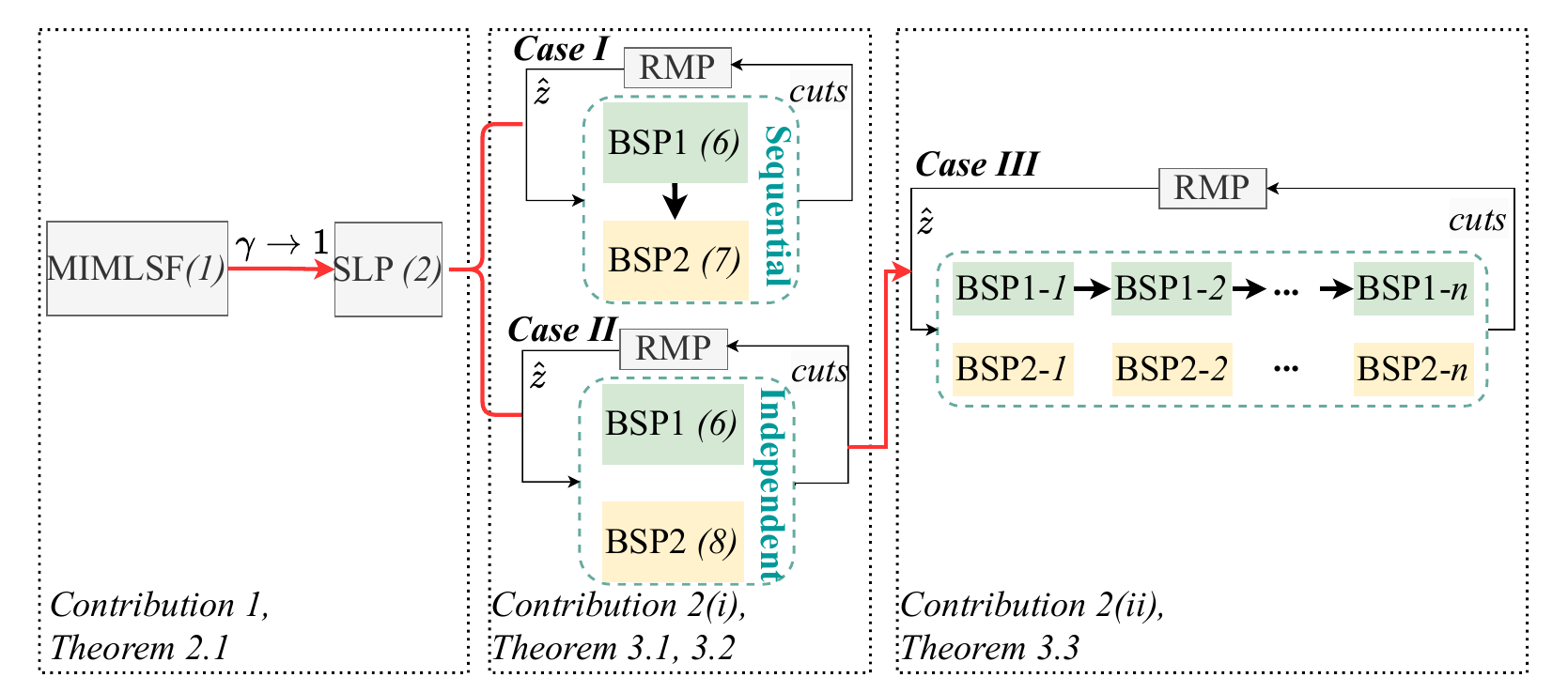}
    \caption{
    Theoretical Contributions. The equation numbers in parentheses denote the complete mathematical formulation of the optimization model. Case \texorpdfstring{\uppercase\expandafter{\romannumeral1}}{I}: general formulation; Case \texorpdfstring{\uppercase\expandafter{\romannumeral2}}{II}: coefficients in Problem~\eqref{eq:1} satisfy \(c_{x1}^T = c_1^T, c_{xi}^T = 0,  \forall i \in [n]^+\); Case \texorpdfstring{\uppercase\expandafter{\romannumeral3}}{III}: Case \texorpdfstring{\uppercase\expandafter{\romannumeral2}}{II} without the \textit{upper dual complicating constraints}~\eqref{eq:1:u4}. RMP: Relaxed Master Problem; BSP: Benders Subproblem.}
    \label{fig:structure}
\end{figure}

The novel contributions of this paper are threefold, and for clarity, Fig.~\ref{fig:structure} presents the theoretical novelty of this paper and each block in Fig.~\ref{fig:structure} corresponds to a key theoretical contribution:
\begin{enumerate}
    \item \textit{Single-level approximation of the MIMLSF problem}: We investigate the the Mixed-Integer Multi-Level problems with Sequential Followers (MIMLSF) (i.e., a $n$+1-level mixed-integer problem with $n$ sequential followers problem), as presented in Fig.~\ref{fig:framework}(d). We propose a lexicographic optimization combined with a weighted-sum approach to transform the MIMLSF problem into a single-level problem (SLP), which asymptotically approximates the original MIMLSF problem as the scaling factor $\gamma$ approaches 1, defined in Section~\ref{section2} (see the first block of Fig.~\ref{fig:structure}). 
    \item \textit{Dedicated Benders subproblem decomposition for the SLP}: We propose a dedicated Benders decomposition approach for the single-level MIMLSF approximation problem (i.e., SLP). The key contributions are: (i) proving that the complex Bender Subproblem (BSP) can be separated into two more tractable problems (BSP1 and BSP2), which can be solved either sequentially or independently, depending on the upper-level objectives (see the second block of Fig.~\ref{fig:structure}); and (ii) demonstrating that these separated BSPs can be further decomposed into $n$ problems while addressing numerical challenges from the scaling factor $\gamma$ (see the third block of Fig. \ref{fig:structure}).
    \item \textit{A practical MIMLSF case study on an integrated electricity-gas-carbon market framework}: We demonstrate our MIMLSF approach through a four-level Unit-Commitment with Gas and Carbon Awareness (UCGCA) case study in the Northeastern United States networks. By incorporating economic feedback from gas and carbon markets, the UCGCA model makes improved UC decisions that mitigate gas price spikes and prevent GFPP financial losses while reducing total operational costs compared to existing approaches. Computational results show that our proposed dedicated Benders decomposition method significantly outperforms direct solution by the Gurobi solver across all test instances.

\end{enumerate}

The rest of the paper is organized as follows: Section~\ref{section2} introduces the MIMLSF model and we prove that it can be asymptotically approximated as a SLP. Section~\ref{section3} presents the dedicated Benders decomposition with further decomposition on the Benders subproblems for the SLP. Section~\ref{section4} presents the UCGCA formulation and results for the UCGCA case studies are presented in Section~\ref{section5}. Lastly, Section~\ref{section6} concludes the paper.

\section{Single-Level Approximation of the MIMLSF Problem}\label{section2}

In this section, we present the general MIMLSF formulation and prove its asymptotic approximation to a single-level problem. This section corresponds the first block of Fig.~\ref{fig:structure}. In what follows, the notation $[n]$ represents the set $\{1, \ldots, n\}$, and $[n]^+$ denotes the set $\{2, \ldots, n\}$ for an integer $n \geq 2$, where $n$ indicates the number of sequential problems. The symbol $\mathbf{1}$ represents the vector of ones. Square brackets around variables indicate the dual variables associated with each constraint.

The MIMLSF problem, characterized as an $n$+1-level structure with $n$ sequential followers (where `1' is the upper level), is presented in Problem~\eqref{eq:1} and illustrated in Fig.~\ref{fig:framework}(d). To improve readability, the problem at level $i$+1 is referred to as the $i^{th}$ lower-level problem, where $i\in[n]$ indicates its sequential position. The term `\textit{lower-level}' is used to denote the position of $n$ sequential problems within the hierarchy of decision-making. We use boldface letters to represent vectors of variables, such as $\bm{x}_i$ and $\bm{y}_i$, where $\bm{x}_i$ indicates the primal variable vector of the $i^{th}$ lower-level problem, and $\bm{y}_i$ represents the dual variable vector:

\begin{subequations}
\centering
    \begin{align}
        &\textbf{MIMLSF:} \min_{\substack{\bm{z}\in\{0,1\}^{m},\\ \bm{x}_{1}\geq 0, \bm{y}_{1}\geq 0,...,\\\bm{x}_{n}\geq 0, \bm{y}_{n}\geq 0}}
         c_{z}^{T} \bm{z} + \sum_{i=1}^n c_{xi}^{T} \bm{x}_{i}\label{eq:1:uo} \\
        & \text{s.t.} \quad \bm{z} \in \mathcal{Z}, \bm{z} \in\{0,1\}^{m} \label{eq:1:u1} \\
        &  \qquad H_{xi}\bm{z} +G_{xi}\bm{x}_{i}\geq b_{xi}, \quad \forall i \in[n]\quad \label{eq:1:u2} \\
        &  \qquad H_{yi}\bm{z} +G_{yi}\bm{y}_{i}\geq b_{yi}, \quad \forall i \in[n]\quad \label{eq:1:u3} \\
        &  \qquad R_y\bm{z}  +\sum_{i=1}^n S_{yi}\bm{y}_{i}\geq q_y, \label{eq:1:u4} \\
        & \qquad(\bm{x}_{1}, \bm{y}_{1},\dots,\bm{x}_{n}, \bm{y}_{n}) =\argmin c_{1}^{T} \bm{x}_{1}\label{eq:1:o1} \\
        & \hspace{3cm}\text{s.t.} \quad A_{1} \bm{x}_{1} + D_{1} \bm{z}  \geq b_{1} \label{eq:1:p1} \\
        & \hspace{4cm} (\bm{x}_{2}, \bm{y}_{2},\dots,\bm{x}_{n}, \bm{y}_{n}) =\argmin c_{2}^{T} \bm{x}_{2}\label{eq:1:o2} \\
        & \hspace{6cm}\text{s.t.} \quad A_{2} \bm{x}_{2} + B_{2} \bm{x}_{1} +D_{2} \bm{z} \geq b_{2}\label{eq:1:p2} \\
        & \hspace{6.8cm} \dots\dots\dots\dots \notag \\
        & \hspace{8cm} \left(\bm{x}_{n},\bm{y}_{n}\right)=\argmin c_{n}^{T} \bm{x}_{n} \label{eq:1:on} \\
        & \hspace{8.5cm}\text{s.t.} \quad A_{n} \bm{x}_{n} + B_{n} \bm{x}_{n-1} +D_{n} \bm{z} \geq b_{n} \label{eq:1:pn}
    \end{align}
    \label{eq:1}
\end{subequations}

In the upper-level problem, the decision variables are modeled as an $m$-dimensional binary vector, $\bm{z}$, where $\bm{z} \in \{0,1\}^m$. The feasible region, denoted as $\mathcal{Z}$, encompasses all binary vectors of dimension $m$, where each component $z_i$ of the vector $\bm{z}$ can take a value of either 0 or 1. Furthermore, the vectors $\bm{x}_i \in \mathbb{R}^{n_{xi}}$ and $\bm{y}_i \in \mathbb{R}^{n_{yi}}$ are defined for each $i \in [n]$, representing the primal and dual variables of the $i^{th}$ lower-level problem, respectively. The mathematical representation of~\eqref{eq:1} with appropriate dimensions for any specific problem can be used to derive the vectors $(c_z, c_{xi},b_{xi},b_{yi},q_y,c_{i},b_i)$ and matrices $(H_{xi},G_{xi}, H_{yi},G_{yi},R_{y},S_{yi},A_i, B_{i^+},D_i)$, $\forall i \in [n]$, $\forall i^+ \in [n]^+$. The upper-level constraints~\eqref{eq:1:u2} and~\eqref{eq:1:u3} are defined as \emph{upper primal coupling constraints} and \emph{upper dual coupling constraints} for each $i \in [n]$, respectively. Moreover, constraint~\eqref{eq:1:u4} is defined as \emph{upper dual complicating constraints} as it accounts for the cumulative impact of the dual variables $\bm{y}_i$ from all lower-level problems. 

The binary decision variables $\bm{z}$ from the upper-level problem are passed onto $n$ sequential lower-level constraints, as suggested in constraints~\eqref{eq:1:p1},~\eqref{eq:1:p2}~\eqref{eq:1:pn}. In addition, the \emph{nomination} determined in the $(i-1)^{th}$ lower-level problem will be fed into the $i^{th}$ (next) lower-level problem and is modeled by the constraints for $i^{th}$ (next) lower-level problem $A_{i} \bm{x}_{i} + B_{i} \bm{x}_{i-1} +D_{i} \bm{z} \geq b_{i}$ where $i\in[n]^+$. Sequential followers framework requires that the $(i-1)^{th}$ lower-level problem does not anticipate the solutions of the $i^{th}$ (next) lower-level problem $(\bm{x}_{i},\bm{y}_{i})$ since constraints and objective function of the $(i-1)^{th}$ lower-level problem are not dependent on the variables of the $i^{th}$ lower-level problem. 


Without loss of generality, we present the lower-level problems in~\eqref{eq:1} with linear constraints. However, the models can be extended to convex lower-level problems (e.g., second-order cone programming (SOCP)) by replacing the linear constraints with appropriate conic forms. This extension is particularly relevant for applications in radial distribution networks \citep{SAVELLI2021102450,XIA2025124721} and gas network problems \citep{8844828}.
We make the following assumptions throughout this paper:
\begin{assumption}\label{assump3}
The following problem is feasible and bounded.
    \begin{align}
&\underset{\substack{\bm{z} \in\{0,1\}^{m},\bm{z}  \in \mathcal{Z},\\\bm{x}_{1}\geq 0, \bm{y}_{1}\geq 0,..,\\\bm{x}_{n}\geq 0, \bm{y}_{n}\geq 0}}{\min } \quad  c_z^{T} \bm{z} +\sum_{i=1}^n c_{xi}^{T} \bm{x}_{i}\\
& \text{s.t.} \quad Eqs.~\eqref{eq:1:u2},~\eqref{eq:1:u3},~\eqref{eq:1:u4},~\eqref{eq:1:p1},~\eqref{eq:1:p2},~\eqref{eq:1:pn}\text{: All primal constraints}\notag 
    \end{align}
\end{assumption}
Assumption~\ref{assump3} provides a lower bound on the optimal objective value of Problem~\eqref{eq:1} by relaxing the optimality of the sequential lower-level problems. 

\begin{assumption}\label{ass1}
For every upper-level decision $\hat{\bm{z}}$, Slater’s constraint qualification holds for the convex sequential lower-level problems \eqref{eq:1:o1}--\eqref{eq:1:pn}.
\end{assumption}
Assumption \ref{ass1} enables a dual-based reformulation of Problem~\eqref{eq:1} through strong duality of the sequential lower-level problems, providing one approach to obtaining global optimal solutions. This can be accomplished by incorporating the primal feasibility constraints, along with the dual feasibility conditions and the strong duality condition, into the problem formulation.
\color{black}
\begin{theorem}\label{theorem1}
The MIMLSF problem \eqref{eq:1} can be asymptotically approximated by the following single-level problem (SLP)~\eqref{eq:2} for some $\gamma \in (0,1)$. Moreover, when $\gamma \rightarrow 1$, the optimal solution of problem \eqref{eq:2} converges to the optimal solution of problem \eqref{eq:1}. In SLP~\eqref{eq:2}, the scaling factor applied for the $i^{th}$ lower-level/sequential problem is denoted by ${\gamma}_i$, with ${\gamma}_1=\gamma$, ${\gamma}_2=\gamma(1-\gamma)$, and continuing as ${\gamma}_n=(1-\gamma)^{n-1}$. The McCormick relaxation technique is applied to linearize the terms $\bm{y}_{i}^{T}D_{i}\bm{z}$ with auxiliary variables $\bm{s}_{yi}^{T}\mathbf{1} = \bm{y}_{i}^{T}D_{i}\bm{z}$ with additional constraints~\eqref{eq:2:linear}. The terms enclosed in squared brackets are dual variables.
\begin{subequations}
    \begin{align}
        &\textbf{SLP: }\min_{\substack{\bm{z}\in\{0,1\}^{m},\\ \bm{x}\geq 0, \bm{y}\geq 0,\\\bm{u}_{x}\geq 0, \bm{u}_{y}\geq 0,\\\bm{u}\geq 0,\bm{s}_{y}\geq 0,\\ \bm{w}\geq 0,\bm{v}_{y}\geq 0}} \gamma c_z^{T} \bm{z}  + \sum_{i=1}^n \gamma c_{xi}^{T} \bm{x}_{i} 
        \label{eq:2:uo} \\
        & \text{s.t.} \quad \bm{z}  \in \mathcal{Z}, \bm{z} \in\{0,1\}^{m}  \label{eq:2:u1} \\
         &  \qquad H_{xi}\bm{z}  +G_{xi}\bm{x}_{i}\geq b_{xi},\quad  \forall i \in[n] \quad\left[\bm{u}_{xi}\geq 0\right]  \label{eq:21:u2} \\
        &  \qquad  \gamma_i H_{yi}\bm{z}  +G_{yi}\bm{y}_{i}\geq \gamma_i b_{yi},\quad \forall i \in[n]\quad\left[\bm{u}_{yi}\geq 0\right] \label{eq:2:u3} \\
        &   \qquad  R_{y}\bm{z}  +\sum_{i=1}^n S_{yi}\frac{\bm{y}_{i}}{\gamma_i}\geq q_{y}, \quad  \forall i \in[n]\quad \left[\bm{u}\geq 0\right] \label{eq:2:u4} \\
        &  \qquad  A_{1} \bm{x}_{1} + D_{1} \bm{z}  \geq b_{1}, \quad \left[\bm{y}_1\geq 0\right] \label{eq:2:p1} \\
        &  \qquad A_{i} \bm{x}_{i} + B_{i} \bm{x}_{i-1} +D_{i} \bm{z} \geq b_{i},\quad  \forall i \in[n]^+  \quad \left[\bm{y}_i\geq 0\right] \label{eq:2:p2}\\
        &  \qquad \bm{y}_{i}^{T}A_{i} +\bm{y}_{i+1}^{T}B_{i+1}  \leq\gamma_i c_{i}^T,\quad  \forall i \in[n-1] \quad \left[\bm{x}_{i}\geq 0\right] \label{eq:2:d1} \\
        &  \qquad \bm{y}_{n}^{T}A_{n}\leq \gamma_n c_{n}^T,\quad \left[\bm{x}_{n}\geq 0\right] \label{eq:2:d2}\\
        & \qquad \sum_{i=1}^n\bigg(\bm{y}_{i}^{T}b_i-\bm{s}_{yi}^{T}\mathbf{1}\bigg)\geq \sum_{i=1}^n \gamma_i c_i^T \bm{x}_{i} \quad\left[\bm{w}\geq 0\right]\label{eq:2:sd}\\
        &  \qquad Y_{yi}\bm{y}_{i}+ K_{yi}\bm{s}_{yi}\geq {\gamma}_i(e_{yi}+E_{yi}\bm{z}),\quad \forall i\in [n] \quad \left[\bm{v}_{yi}\geq 0\right]\label{eq:2:linear}
    \end{align}
    \label{eq:2}
\end{subequations}
\end{theorem}
\begin{proof}
Please refer to Appendix A of the supplementary material.
\end{proof}

\section{Dedicated Benders Decomposition with Multi-level Subproblem Separability}\label{section3}
In this section, we propose a dedicated Benders decomposition approach with enhanced subproblem separability to effectively solving the single-level problem~\eqref{eq:2}. This section relates to the second and the third blocks of Fig.~\ref{fig:structure}. While standard solvers may be capable of solving the single-level approximation of a tri-level model with sequential middle and lower levels of moderate complexity (i.e., $n=2$), as indicated in \citet{XIA2025124721,10151915}, they might struggle when addressing more intricate or larger-scale problems. In this case, Benders decomposition offers an alternative for solving the large-scale complex SLP~\eqref{eq:2}. 

Benders decomposition involves iteratively solving a Relaxed Master Problem (RMP) with binary variables and a Benders Subproblem (BSP) with continuous variables to find the optimal solution of the original problem by introducing cuts based on the feasibility and optimality of subproblems until the upper and lower bounds converge sufficiently. For a detailed review of Benders decomposition, see \citet{RAHMANIANI2017801}. To facilitate the application of Benders decomposition,~\eqref{eq:2} is initially reformulated as follows with a given $\bm{\hat{z}}$:\begin{subequations}
    \begin{align}
        &\min_{\bm{\hat{z}}}
      \quad  \gamma c_z^{T} \bm{\hat{z}} + f(\bm{\hat{z}}) \label{eq:3:1} \\
        & \text{s.t.} \quad \bm{\hat{z}}\in \mathcal{Z}, \bm{\hat{z}} \in\{0,1\}^{m}   \label{eq:3:2}
    \end{align}
    \label{eq:3}
\end{subequations}
where $f(\bm{\hat{z}})$ is defined as 
\begin{subequations}
    \begin{align}
        & f(\bm{\hat{z}}):=\min_{\substack{ \bm{x}\geq 0, \bm{y}\geq 0,\\\bm{u}_{x}\geq 0, \bm{u}_{y}\geq 0,\bm{u}\geq 0,\\\bm{s}_{y}\geq 0, \bm{w}\geq 0,\bm{v}_{y}\geq 0}}
        \sum_{i=1}^n \gamma c_{xi}^{T} \bm{x}_{i} \label{eq:4:1} \\
        & \text{s.t.} \quad \text{Eq. \eqref{eq:21:u2}}- \text{Eq. \eqref{eq:2:linear}}\textit{ with $\bm{z}$ replaced by $\hat{\bm{z}}$}  \label{eq:4:2} 
    \end{align}
    \label{eq:4}
\end{subequations}
For a guess $\bm{\hat{z}}$, then the corresponding BSP of Problem~\eqref{eq:2} is derived by taking the dual of Problem~\eqref{eq:4}:
\begin{subequations}
    \begin{align}
        &\textbf{BSP:}\max_{\substack{ \bm{x}\geq 0, \bm{y}\geq 0, \bm{v}_y\geq 0, \\
        \bm{w}\geq 0, \bm{u}\geq 0, \bm{u}_x\geq 0, \\
        \bm{u}_y\geq 0, \bm{s}_y\geq 0 }}
         \sum_{i=1}^n \left(\bm{y}_{i}^{T}(b_i-D_{i}\bm{\hat{z}})+\bm{u}_{xi}^T(b_{xi}-H_{xi}\bm{\hat{z}})\right)\notag \\
         &\qquad\qquad\qquad-\bigg[\sum_{i=1}^n\left(\gamma_i c_{i}^T \bm{x}_{i} -\gamma_i \bm{u}_{yi}^T(b_{yi}-H_{yi}\bm{\hat{z}}) -\gamma_i v_{yi}^T(e_{yi}+E_{yi}\bm{\hat{z}})\right)-\bm{u}^T(q_{y}-R_{y}\bm{\hat{z}}) \bigg]\label{eq:5:o1}\\
        & \text{s.t.} \quad \bm{y}_{i}^{T}A_{i} +\bm{y}_{i+1}^{T}B_{i+1}+ G_{xi}^T\bm{u}_{xi} \leq \gamma_i c_{i}^T \bm{w} + \gamma c_{xi}^T,\quad \forall i \in [n-1]\label{eq:5:y1}\\
        &\qquad \bm{y}_{n}^{T}A_{n}+ G_{xn}^T\bm{u}_{xn} \leq \gamma_n c_{n}^T \bm{w}+\gamma c_{xn}^T,\label{eq:5:y2}\\
        &\qquad A_{1} \bm{x}_{1} -G_{y1}^T\bm{u}_{y1} -Y_{y1}^T\bm{v}_{y1}-S_{y1}^T\bm{u}/\gamma_1\geq b_{1}\bm{w}, \label{eq:5:x1} \\
        & \qquad A_{i} \bm{x}_{i} + B_{i} \bm{x}_{i-1} -G_{yi}^T\bm{u}_{yi} -Y_{yi}^T\bm{v}_{yi}-S_{yi}^T\bm{u}/\gamma_i\geq b_{i}\bm{w},\quad \forall i \in [n]^+ \label{eq:5:x2}\\
        & \qquad K_{yi}^T\bm{v}_{yi}\leq \bm{w}, \quad \forall i\in[n]\label{eq:5:v1}
    \end{align}
    \label{eq:5}
\end{subequations}

While the single-level problem \eqref{eq:2} can be iteratively solved using the RMP and the BSP \eqref{eq:5}, the BSP \eqref{eq:5} is computationally intensive due to its inclusion of both primal variables $\bm{x}_{i}$ related constraints (see \eqref{eq:5:x1} and \eqref{eq:5:x2}) and dual variables $\bm{y}_{i}$ related constraints (see \eqref{eq:5:y1} and \eqref{eq:5:y2}). Additionally, the complexity of the BSP \eqref{eq:5} increases with the number of lower-level problems $n$. Note that the strong duality condition \eqref{eq:2:sd} serves as a complicating constraint in the SLP \eqref{eq:2}, linking $n$ sequential lower-level problems. Consequently, its corresponding dual variable $\bm{w}$ acts as a linking variable in the BSP \eqref{eq:5}, see the right-hand side of each constraint in~\eqref{eq:5}. In the following subsections, we aim to present solution techniques for the complex BSP~\eqref{eq:5}, addressing the following three different cases that involve different upper-level objectives~\eqref{eq:1:uo} and constraint requirements on~\eqref{eq:1:u4}, as shown in Table~\ref{tab:cases_summary}:
\begin{table}[h]
\centering
\caption{Summary of Cases and their Descriptions}
\label{tab:cases_summary}
\renewcommand{\arraystretch}{1}
\begin{tabular}{cccc}
\hline
Case & Upper-level Objective Requirements& Constraint Requirements & Extra Notes \\ \hline
I & \(c_z^{T} \bm{z} + \sum_{i=1}^n c_{xi}^{T} \bm{x}_{i}\)  & None & None \\ 
II & \(c_z^{T} \bm{z} + c_1^{T} \bm{x}_1\) & None & Special case of Case I \\ 
III & \(c_z^{T} \bm{z} + c_1^{T} \bm{x}_1\)  & No \eqref{eq:1:u4} & Special case of Case II \\ \hline
\end{tabular}
\end{table}


\subsection{Benders Subproblem Decomposition for Case  \texorpdfstring{\uppercase\expandafter{\romannumeral1}}{I}}\label{section3a}
This subsection discusses the methodologies outlined in the upper second blocks of Fig.~\ref{fig:structure} for Case  \texorpdfstring{\uppercase\expandafter{\romannumeral1}}{I}. To further decompose the Benders subproblem~\eqref{eq:5}, this paper adopts the Bender's Separation Scheme as proposed in \citet{doi:10.1287/ijoc.2021.1128}, extending its application to multi-level problems. The idea is to prove that the complex BSP~\eqref{eq:5} can be solved by solving two relatively tractable problems sequentially. In other words, Benders cuts of~\eqref{eq:2} can be obtained by solving two more tractable problems.
\begin{theorem}\label{theorem2}
Benders subproblem~\eqref{eq:5} can be solved by two relatively tractable problems sequentially. First, solve problem~\eqref{eq:6}
\begin{subequations}
    \begin{align}
        &\textbf{BSP1:} \qquad \min_{\substack{ \bm{x}\geq 0,\bm{u}_y\geq 0,\\ \bm{v}_y\geq 0,\bm{u}\geq 0}}
        \sum_{i=1}^n \left(\gamma_i c_{i}^T \bm{x}_{i}-\gamma_i \bm{u}_{yi}^T(b_{yi}-H_{yi}\bm{\hat{z}})-\gamma_i \bm{v}_{yi}^T(e_{yi}+E_{yi}\bm{\hat{z}}) \right)-\bm{u}^T(q_{y}-R_{y}\bm{\hat{z}})\label{eq:6:o1} \\
        & \text{s.t.}\quad A_{1} \bm{x}_{1} -G_{y1}^T\bm{u}_{y1} -Y_{y1}^T\bm{v}_{y1}-S_{yi}^T\bm{u}/\gamma_1\geq b_{1} \label{eq:6:x1} \\
        & \qquad A_{i} \bm{x}_{i} + B_{i} \bm{x}_{i-1} -G_{yi}^T\bm{u}_{yi} -Y_{yi}^T\bm{v}_{yi}-S_{yi}^T\bm{u}/\gamma_i\geq b_{i},\quad\forall i \in [n]^+ \label{eq:6:x2}\\
        & \qquad K_{yi}^T\bm{v}_{yi}\leq \bm{1}, \quad i\in[n]\label{eq:6:v1}
    \end{align}
    \label{eq:6}
\end{subequations}
Then, solve problem~\eqref{eq:7} where $\mathfrak{D}_6$ is the optimal value of problem~\eqref{eq:6} if it has a finite optimum, $\infty$ otherwise (i.e. by fixing $\bm{w}=0$):
\begin{subequations}
    \begin{align}
        &\textbf{BSP2:} \qquad \max_{\substack{ \bm{y}\geq 0,\bm{u}_{x}\geq 0,\\\bm{w}\geq 0}}
         \sum_{i=1}^n \left(\bm{y}_{i}^{T}(b_i-D_{i}\bm{\hat{z}})+\bm{u}_{xi}^T(b_{xi}-H_{xi}\bm{\hat{z}})\right)-\bm{w}\mathfrak{D}_6 \label{eq:7:o1} \\
        & \text{s.t.} \quad\bm{y}_{i}^{T}A_{i} +\bm{y}_{i+1}^{T}B_{i+1}+ G_{xi}^T\bm{u}_{xi} \leq \gamma_i c_{i}^T\bm{w} + \gamma c_{xi}^T ,\quad \forall i \in [n-1]\label{eq:7:y1} \\
        & \qquad \bm{y}_{n}^{T}A_{n}+ G_{xn}^T\bm{u}_{xn}\leq \gamma_n c_{n}^T\bm{w} + \gamma c_{xn}^T \label{eq:7:yn}
    \end{align}
    \label{eq:7}
\end{subequations}
\end{theorem}
\begin{proof}
Please refer to Appendix B of the supplementary material.
\end{proof}
Theorem~\ref{theorem2} suggests that the BSP~\eqref{eq:5} can be solved by two tractable problems~\eqref{eq:6} and~\eqref{eq:7}, and the RMP with feasibility and optimality cuts are given in Corollary 1 and the dedicated Benders decomposition algorithm is presented in Algorithm 1. For further details on Corollary 1 and Algorithm 1, please see Appendix C in the supplementary material.


\subsection{Benders Subproblem Decomposition for Case  \texorpdfstring{\uppercase\expandafter{\romannumeral2}}{II}}\label{section3b}
In this section, we examine a specific instance of Case  \texorpdfstring{\uppercase\expandafter{\romannumeral1}}{I}, referred to as Case \texorpdfstring{\uppercase\expandafter{\romannumeral2}}{II}, which allows for a significantly strong alternative to Theorem~\ref{theorem1}. This subsection details the methodologies presented in the lower second blocks of Fig.~\ref{fig:structure} for Case \texorpdfstring{\uppercase\expandafter{\romannumeral2}}{II}. We examine a special configuration of the upper-level objectives represented by $ c_z^T \bm{z} +  c_1^T \bm{x}_1$, i.e., \(c_{x1}^T = c_1^T\), \(c_{xi}^T = 0\) for all \(i \in [n]^+\). These conditions can be approximated as \(c_{xi}^T = \gamma_i c_{i}^T\) as \(\gamma\) approaches 1. This special configuration is adopted by the UCGCA model introduced in the next section. We denote the BSP under this special condition is BSP'~\eqref{eq:5}' where the RHS of constraints~\eqref{eq:5:y1} and~\eqref{eq:5:y2} becomes $\gamma_i c_{i}^T (\bm{w} +1)$, $\forall i\in[n]$. Building on \citet{doi:10.1287/ijoc.2021.1128}, we prove that, unlike the approach specified in Theorem~\ref{theorem1} which involves solving two problems \textit{sequentially}, the BSP'~\eqref{eq:5}' can be addressed through two \textit{independent} problems.
\begin{theorem}\label{theorem3}
BSP'~\eqref{eq:5}' can be solved by two relatively tractable problems~\eqref{eq:6} and~\eqref{eq:8} independently:
\begin{subequations}
    \begin{align}
        &\textbf{BSP2:} \qquad \max_{\substack{ \bm{y}\geq 0,\bm{u}_{x}\geq 0}}
         \sum_{i=1}^n \bm{y}_{i}^{T}(b_i-D_{i}\bm{\hat{z}})+\bm{u}_{xi}^T(b_{xi}-H_{xi}\bm{\hat{z}})\label{eq:8:o1} \\
        & \text{s.t.}  \hspace{0.5em}\bm{y}_{i}^{T}A_{i} +\bm{y}_{i+1}^{T}B_{i+1}+ G_{xi}^T\bm{u}_{xi} \leq \gamma_i c_{i}^T ,\forall i\in [n-1] \label{eq:8:y1} \\
        & \qquad \bm{y}_{n}^{T}A_{n}+ G_{xn}^T\bm{u}_{xn}\leq \gamma_n c_{n}^T \label{eq:8:yn}
    \end{align}
    \label{eq:8}
\end{subequations}
\end{theorem}
\begin{proof}
Please refer to Appendix D of the supplementary material.
\end{proof}
The proof of Theorem~\ref{theorem3} implies that the Benders cuts for the SLP~\eqref{eq:2} with $c_{xi}^T = \gamma_i c_{i}^T$ conditions can be obtained by solving Problems~\eqref{eq:6} and~\eqref{eq:8} \emph{independently} and comparing their objective values. This simplifies the Benders cut generation algorithm as described in Algorithm 2 and the RMP is presented in Corollary 2. For further details on Corollary 2 and Algorithm 2, please see Appendix E in the supplementary materials.

\subsection{Benders Decomposed Subproblems with Further Multi-level Separation for Case  \texorpdfstring{\uppercase\expandafter{\romannumeral3}}{III}}\label{section3c}
In this section, we investigate a particular instance of Case  \texorpdfstring{\uppercase\expandafter{\romannumeral2}}{II}, referred to as Case  \texorpdfstring{\uppercase\expandafter{\romannumeral3}}{III}, that introduces advanced features building on Theorem~\ref{theorem3}. This subsection examines the approach suggested in the third blocks of Fig.~\ref{fig:structure}. While Theorem~\ref{theorem3} demonstrates that the computation of BSPs with the special upper-level objectives can be theoretically separated into two more tractable problems~\eqref{eq:6} and~\eqref{eq:8} (i.e., decomposed into a primal-related problem~\eqref{eq:6} and a dual-related problem~\eqref{eq:8}), the complexity of solving these problems still increases with the number of lower-level problems $n$. Additionally, the use of the weighted-sum method introduces scaling factors $\gamma_i$ to the objective functions for variables $(\bm{x}, \bm{u}_x, \bm{u}_y)$ in the subproblems defined by~\eqref{eq:6} and the dual variables $\bm{y}$ in~\eqref{eq:8} are scaled accordingly. This scaling may introduce numerical issues, particularly as $\gamma$ approaches 1 (and therefore $(1-\gamma) \rightarrow0$) and as the number of lower-level problems increases. To address this issue, we propose the following Theorem:
\begin{theorem}
\label{theorem3.3}
BSP1~\eqref{eq:6} can be further decomposed into $n$ problems that can be solved \emph{sequentially} if there are no \textit{upper dual complicating constraints}~\eqref{eq:1:u4} and BSP2~\eqref{eq:8} can be further decomposed into $n$ problems and solved \emph{in parallel}.
\end{theorem}
\begin{proof}
Please refer to Appendix F of the supplementary material.
\end{proof}



Theorem \ref{theorem3.3} demonstrates that when solving BSP~\eqref{eq:6} with a fixed upper-level decision $\hat{\bm{z}}$ from the master problem, only connecting variables (e.g. $\bm{x}_i$, $i\in[n-1]$) need to be exchanged between sequential lower-level problems, limiting the information flow between them. Without it, solving the single-level problem~\eqref{eq:2} would require directly incorporating variables and constraints from $n$ lower-level agents into BSP1~\eqref{eq:6} and BSP2~\eqref{eq:8}. Such direct incorporation would require scaling factors to capture the sequential relationships between lower-level problems, leading to significant numerical scalability challenges. Instead, Theorem~\ref{theorem3.3} enables a more scalable solution approach. It allows BSP~\eqref{eq:6} and BSP~\eqref{eq:8} to be solved either sequentially or independently by leveraging the interpretation of scaling factors across $n$ problems. This eliminates the computational burden and avoids scalability issues of solving BSP~\eqref{eq:6} and BSP~\eqref{eq:8} as $n$ coupled lower-level problems.

\section{The Four-level Unit Commitment with Gas and Carbon Awareness Model}\label{section4}
In this section, we provide an application of the MIMLSF model by proposing a four-level Unit Commitment with Gas and Carbon Awareness (UCGCA) model, which integrates the electricity, gas, and carbon market framework. The design of the bid-validity constraints for GFPPs and the coupling between three markets are introduced in Section~\ref{section4B}. The relationship between the UCGCA model and the MIMLSF problem is then discussed in Section~\ref{section4C}. The mathematical formulations for UC, economic dispatch (ED), gas market clearing, and carbon market clearing problems are detailed in Appendix G of the supplementary material. The ED problem is also referred to as the electricity market clearing problem.

\begin{figure}[tp]
    \centering
    \includegraphics[width=0.5\columnwidth]{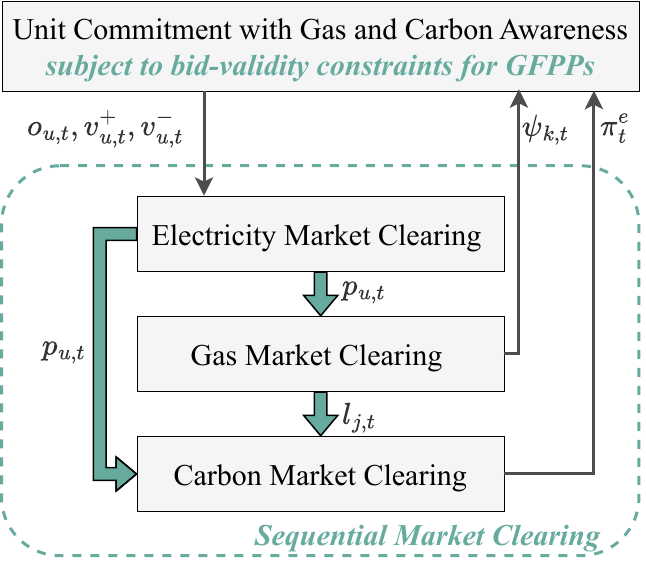}
    \caption{Four-level UCGCA model within an integrated electricity, gas, and carbon market clearing framework.}
    \label{fig:four-level}
\end{figure}
Fig.~\ref{fig:four-level} illustrates the interactions among UC, electricity, natural gas, and carbon market clearing problems. The upper-level UC problem is managed by Independent System Operators (ISOs). The commitment decisions, including the on/off statuses $o_{u,t}$, and indicators for generator $u$ start-ups $v^+_{u,t}$ and shut-downs $v^-_{u,t}$ at time $t$, are fed into the ED problem (i.e., electricity market clearing) to determine the power output $p_{u,t}$. This power schedule $p_{u,t}$ for GFPPs generates a gas demand \( l^{GFPP}_{j,t} \) at the corresponding gas network junction \( j \), which is approximated using the heat rate curve~\eqref{eq:gfppconversion}, as further discussed in Section~\ref{section4B}. The power output \( p_{u,t} \) and the satisfied gas demand \( l_{j,t} \) lead to carbon emissions, enabling participants to trade surplus allowances or purchase additional allowances in the carbon market. A net allowance greater than zero indicates a surplus, allowing participants to sell on the market for a profit; conversely, a deficit requires purchasing allowances, incurring costs. The derived dual solutions, including zonal gas prices \( \psi_{k,t} \) and carbon prices \( \pi^e_t \), are then incorporated into the upper-level bid-validity constraints to ensure only profitable GFPPs are committed. The design of this bid-validity constraint is discussed in Section~\ref{section4B}.

Power system operations fundamentally rely on two sequential processes: UC and ED, which operate at sequential temporal scales \citep{SAKHAVAND2024100089,JointBoard2006SCED}. The day-ahead UC problem, typically formulated as a Mixed-Integer Linear Programming (MILP) model, determines the binary commitment status -- start-up and shut-down decisions of generating units for each hour of the following operating day -- while considering various operational constraints. Subsequently, as the operating period approaches, system operators execute ED closer to real-time to optimize the power output levels of the previously committed resources, ensuring supply-demand balance and adherence to network physical constraints while minimizing overall operational costs.

\subsection{The Physical and Economic Coupling of Three Markets and the Bid Validity Constraints}\label{section4B}
The physical coupling between electric and gas network is expressed through the heat rate curve~\eqref{eq:gfppconversion} for GFPPs, which is approximated as a quadratic function relating gas consumption $l^{GFPP}_{j,t}$ to electricity generation $p_{u,t}$:
\begin{align}
\qquad \qquad \qquad \qquad \qquad l^{GFPP}_{j,t} = \sum_{u \in \mathcal{U}(j) \cap \mathcal{U}^g} H^G_{u,2} p_{u,t}^2 + H^G_{u,1} p_{u,t} + H^G_{u,0}, \quad \forall j \in \mathcal{V}, \forall t \in \mathcal{T}.\label{eq:gfppconversion}
\end{align}
The physical coupling between the electric and gas markets with the carbon market is facilitated through the power generation level \( p_{u,t} \) and the satisfied gas demand \( l_{j,t} \). The latter is computed as the difference between the gas demand profile \( d^g_{j,t} \) and the gas load shed \( q_{j,t} \). Specifically, the emission levels are determined by the Carbon Intensity (CI) factors \( \kappa_{u} \) and \( \kappa_{j} \), which are multiplied by their respective cleared quantities, as shown in Eq.~\eqref{eq:carbon}.
\begin{subequations}
\begin{align}
& \qquad \qquad \qquad \qquad \qquad \qquad \qquad \qquad E_{u,t} = \kappa_{u} p_{u,t}, \quad\forall u \in \mathcal{U}, t \in [T],\label{eq:carbon_con1} \\
& \qquad \qquad \qquad \qquad \qquad\qquad \qquad \qquad E_{j,t} = \kappa_{j} l_{j,t}, \quad \forall j \in \mathcal{V}, t \in [T].\label{eq:carbon_con2}
\end{align}
\label{eq:carbon}
\end{subequations}
The power generation levels of GFPPs significantly influence the load on the gas system and carbon emission levels and, consequently, natural gas prices and carbon prices. These prices, in turn, determine the profitability of GFPPs, which place their bids in the electricity market prior to realizing gas and carbon prices. These dynamics are captured in the bid-validity constraints~\eqref{eq:bidvalidity}, designed to compare the GFPPs’ marginal bids $\rho_{u,t}$ (see the definition and related constraints of $\rho_{u,t}$ in Appendix G.1 of the supplementary material) in the electricity market with the economic realities of the gas and carbon markets. It ensures that their participation remains economically viable and profitable by aligning commitment decisions $o_{u,t}$ with anticipated outcomes from gas and carbon markets. $\alpha_u$ represents the risk aversion level of GFPP $u$, modulating the relationship between the plant’s expected profits and the risks it faces in volatile markets. A lower value of $\alpha_u$ indicates a more risk-averse GFPP, meaning it would prefer to de-commit to avoid potential financial losses.
\begin{align}
\qquad\qquad\alpha_u\rho_{u,t} \geq \big((2p_{u,t}H_{u,2}^G + H_{u,1}^G) \psi_{k,t}+(2y_{u,t}-1)\kappa_{u}\pi^{e}_t \big) o_{u,t},\quad \forall k \in \mathcal{K}, i \in \mathcal{V}(k), u \in \mathcal{U}(i) \cap \mathcal{U}^{g}.\label{eq:bidvalidity}
\end{align}
The right-hand side of Equation~\eqref{eq:bidvalidity} is composed of two main components:

\begin{enumerate}
    \item Marginal Natural Gas Cost $(2p_{u,t}H_{u,2}^G + H_{u,1}^G) \psi_{k,t}$: The term \( (2p_{u,t}H_{u,2}^G + H_{u,1}^G) \) represents the derivative of the heat rate curve~\eqref{eq:gfppconversion}, representing the amount of natural gas needed to generate one additional unit of electricity. This term is multiplied by the zonal gas prices \(\psi_{k,t}\) to represent the marginal cost of GFPP $u$.

    \item Marginal Carbon Trading Cost or Revenue $(2y_{u,t}-1)\kappa_{u}\pi^{e}_t $: When \( y_{u,t} = 1 \), indicating the GFPP with deficit allowances which purchases emission allowances in the carbon market, adding a marginal cost \(\kappa_{u}\pi^{e}_t \) to generate one additional unit of electricity. On the other hand, when \( y_{u,t} = 0 \), the GFPP with surplus allowances sells its carbon allowances, leading to a negative term of \(-\kappa_{u}\pi^{e}_t \), which represents the marginal revenue gained.
\end{enumerate}

In the gas market modeling, dual solutions associated with the flux conservation constraints represents the marginal costs at gas junction $j$. However, the US natural gas spot price is zonal, the zonal natural gas prices $\psi_{k,t}$ at zone $k$ are then computed based on averaging the prices of a subset of junctions \citep{8844828}. In the case study, we examine two distinct gas pricing zones following \citet{8395045,8844828,zonewithline}: the Transco Zone 6 Non NY Zone and the Transco Leidy Zone. The Transco Zone 6 Non NY Zone typically exhibits higher prices due to its location in a major gas consumption area, while the Transco Leidy Zone, situated in the Marcellus Shale production area, generally maintains lower prices due to abundant natural gas supplies. The selection of these zones enables analysis of price dynamics across regions with significant price differentials.

The bid-validity constraint~\eqref{eq:bidvalidity} is added to the upper-level UC problem and it introduces nonlinearity into the model and these terms can be linearized by using the McCormick relaxation technique. The bid-validity constraint~\eqref{eq:bidvalidity} enables GFPPs to hedge against the risks associated with high gas prices and emission costs/revenue. This risk management is vital, ensuring that committed GFPPs consider all relevant costs and potential revenues following their bid submissions in the electricity market. By incorporating this constraint into the UC problem, system operators can improve the financial viability of GFPPs and prevent default or large financial losses. This is particularly crucial during periods of high demand, such as the 2014 polar vortex experienced in the Northeastern United States, to prevent defaults \citep{8844828, PJM2014ColdWeather}.

\subsection{Relationship between the four-level UCGCA model and the MIMLSF problem}\label{section4C}

The UCGCA model is a specific instance of the MIMLSF problem where \(n=3\), corresponding to the three sequential lower-level problems. The UCGCA model features a specialized structure for the dedicated Benders decomposition as outlined in Section~\ref{section3b}: the objective function of the upper-level UC problem includes:
1) \(c_z^T \bm{z}\): costs associated with the binary variable \(\bm{z}\), encompassing no-load and start-up costs,
2) \(c_{x1}^T \bm{x}_1\): costs of selected supply bids from electrical power generating units, which is exactly the objective function of the second-level ED problem (i.e., $c_{x1}^T = c_1^T, c_{xi}^T = 0,  \forall i \in [n]^+$, the upper-level objective is $ c_z^{T} \bm{z} +  c_{1}^{T} \bm{x}_{1}$ and the first lower-level objective is $c_{1}^{T} \bm{x}_{1}$). Therefore, the dedicated Benders decomposition with subproblem separability introduced in Section~\ref{section3b} for Case \uppercase\expandafter{\romannumeral2} can be effectively applied for the UCGCA model.

\section{Real-World Implementation of MIMLSF: Four-Level UCGCA Model in the Northeastern United States}\label{section5}

This section presents a case study on the integration of the IEEE 36-bus Northeastern U.S. bulk electric power system \citep{4560041,8395045} with a multi-company gas transmission network spanning from Pennsylvania to Northeast New England \citep{zonewithline,8844828,8395045}, as depicted in Fig.~\ref{fig:USnetworks}. In Fig.~\ref{fig:USnetworks}, blue connections represent electricity transmission lines, with blue markers indicating the locations and magnitudes of electricity generation and load levels. Similarly, green connections represent natural gas transmission lines, with green markers indicating the locations and magnitudes of gas supply and demand levels. The gas electric coupling nodes are depicted in red circle which specifies the location of GFPPs and the GFPP of the electric power network were linked to the closest natural gas receipt point in the gas system. As discussed in Section~\ref{section4B}, we consider two natural gas pricing zones: Transco Zone 6 Non NY Zone and Transco Leidy Line Zone. Fig.~H.1 in the supplementary material shows the pricing junctions for these zones. Transco Zone 6 Non NY Zone are represented by square markers, whereas the pricing points/junctions for Transco Leidy Line Zone are represented by diamond markers. The pricing points are based on previous studies \citet{8395045,8844828} and the GasPowerModels.jl repository \url{https://github.com/lanl-ansi/GasPowerModels.jl}.

\begin{figure}[th!]
    \centering
    \includegraphics[width=0.8\columnwidth]{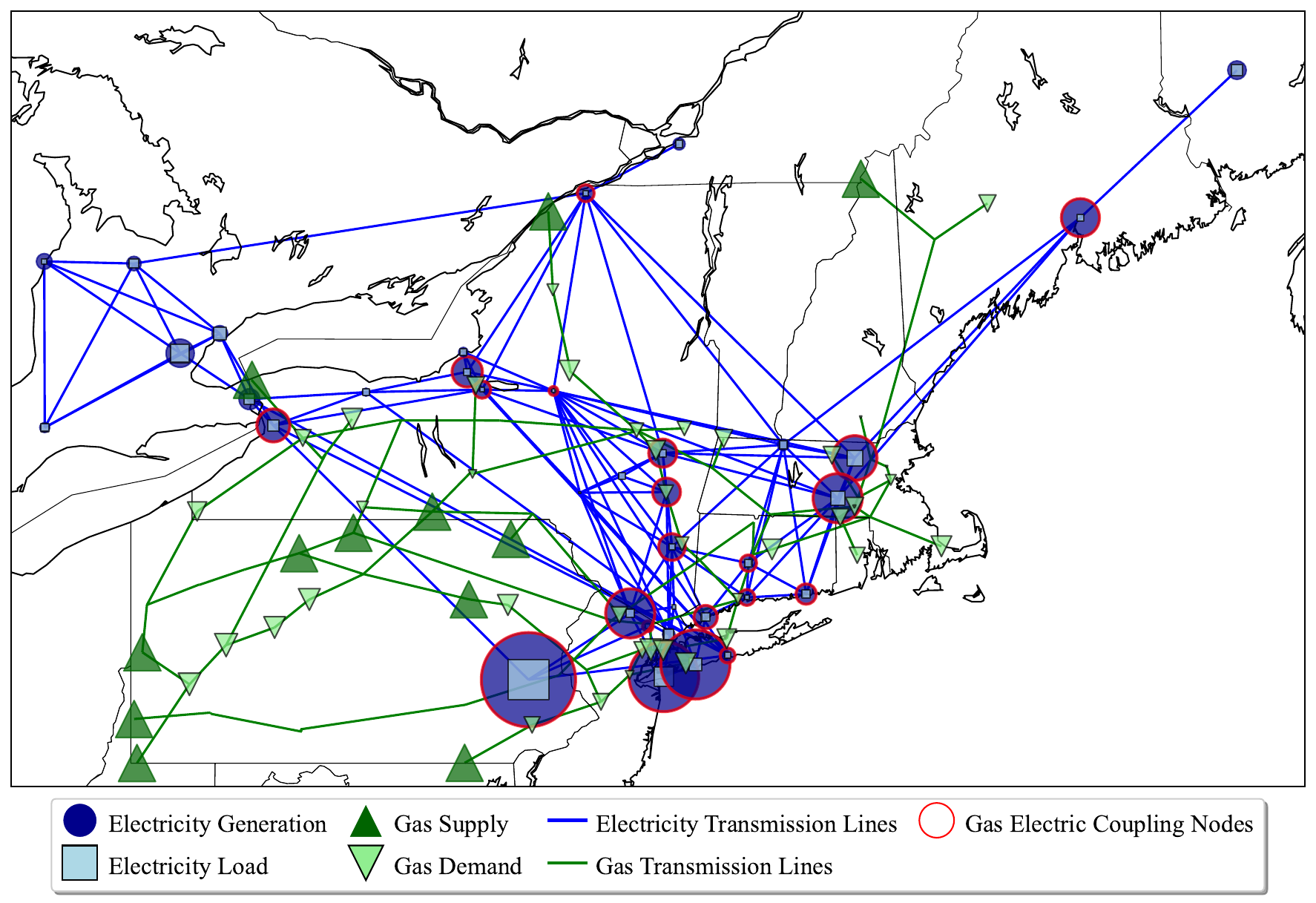}
    \caption{
    Electric and gas transmission networks in the Northeastern United States, depicted with simplified point-to-point connections. Node markers are sized proportionally to maximum supply and load capacities at $(\eta_e,\eta_g)=(1,1)$, with gas system markers shown on a logarithmic scale. GFPPs in the electric network are connected to their nearest natural gas receipt points in the gas network.}
    \label{fig:USnetworks}
\end{figure}


\begin{table}[th!]
    \centering
    \caption{Carbon Intensity of Generators in the Electric Network, sourced from \citet{SAVELLI2022106218}. Fuel types include Oil (O), Coal (C), Open Cycle Gas Turbine (G-O), Combined Cycle Gas Turbine (G-C), Hydro (H), Refuse (R), Nuclear (N), and Others (E).}

    \begin{tabular}{ccccccccc}
    \hline
    \textbf{Fuel Type} & \textbf{O} & \textbf{C} & \textbf{G-O} & \textbf{G-C} & \textbf{H }& \textbf{R} & \textbf{N} & \textbf{E}  \\
    \hline
    CI ($tCO_2/MWh$) & 0.777 & 0.937 & 0.651 & 0.394 & 0 & 0.120 & 0 & 0.300 \\
    \hline
    \end{tabular}%
    
    \label{table:carbon_generator}
\end{table}
To analyze system behavior under different market conditions, we varied both electrical and gas load parameters in our case study. We examined electrical load increases of 20\% and 60\% (i.e., $\eta_e = {1.0,1.2,1.6}$) and gas load increases ranging from 10\% to 130\% (i.e., $\eta_g = {1.0,1.1,\ldots,2.3}$). These parameter variations allowed us to evaluate the economic viability of GFPP under increasing gas costs, assess the effectiveness of bid-validity constraints, and demonstrate the performance advantages of the proposed UCGCA approach. The electric network examined includes 91 generators of various fuel types, each characterized by specific CI values referenced by \citet{SAVELLI2022106218}, as listed in Table~\ref{table:carbon_generator}. The UC data is derived from the RTO Unit Commitment Test System, with further details provided in \citet{8844828, FERC2024}. Bid prices for participants in the carbon market are sampled from normal distributions with standard deviations of \$10/$tCO_2$ and means equal to \$30.24/$tCO_2$. Negative prices are set to zero. The mean value is based on the California Cap-and-Trade Program's carbon allowance prices for Q3 2024 \citep{CARBCapAndTrade2024}. In the absence of specific allowance data, we allocate carbon allowances to each generator at 50\% of their maximum emission potential, calculated as the maximum power output multiplied by the carbon intensity. On the other hand, carbon allowances for each gas load junction are set at 165\% of the firm gas load under standard conditions (i.e., when $\eta_g = 1.65$) multiplied by the carbon intensity (CI) factor for natural gas, which is set at 55 $tCO_2$/mmcf according to \citet{epa_ghg_equivalencies}. The selection of $\eta_g = 1.65$ represents an intermediate stress point on the gas network, positioned in the middle of the tested range ($\eta_g = {1.0, 1.1, \ldots, 2.3}$). The cost associated with gas load shedding is set to \$130/mmBtu \citep{8844828}, while the cost for acquiring additional external carbon allowances is priced at \$50/$tCO_2$. The risk aversion level, $\alpha_u$, is set to 100\%. 

\begin{figure}[th!]
    \centering
    \includegraphics[width=0.8\columnwidth]{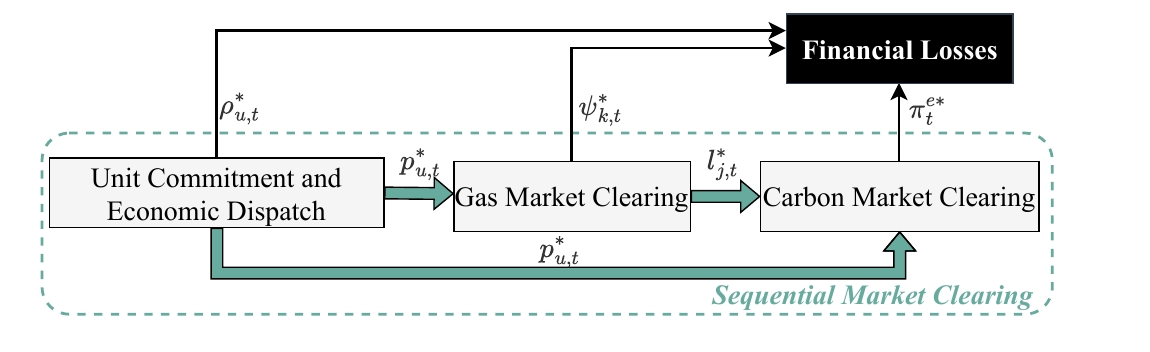}
    \caption{Three-stage sequential Benchmark Model (BM) without GFPPs' bid validity constraints.}
    \label{fig:BM_model}
\end{figure}

To demonstrate the effectiveness of incorporating anticipated gas and carbon market outcomes into the UC problem (i.e., the UCGCA model), we evaluate its performance against a Benchmark Model (BM), as illustrated in Fig.~\ref{fig:BM_model}. The BM is structured as a three-stage sequential optimization problem: initially, the UC and ED are solved to determine the power outputs of GFPPs. These outputs then inform the gas demand induced by GFPPs $l^{GFPP}_{j,t}$ required for gas market clearing, which is calculated using the heat rate curve~\eqref{eq:gfppconversion}. Subsequently, the carbon market is cleared based on actual emissions and submitted bids. Fig.~\ref{fig:BM_model} shows the three-stage BM structure without any bid validity constraints specifically for GFPPs, and invalid bids can be calculated after the three-stage market clearing. Notably, the BM does not account for the economic viability of GFPPs. These plants submit their bids in the electricity market prior to the realization of gas and carbon prices. As a result, they may incur financial losses from invalid bids, especially under conditions of high stress in gas and electricity networks. These losses are quantified by calculating the sum of the difference between the violated GFPPs' marginal bids $\rho^*_{u,t}$ times the risk aversion level $\alpha_u$ in the electricity market and the realized costs of gas and carbon $(2p^*_{u,t}H_{u,2}^G + H_{u,1}^G) \psi^*_{k,t}+(2y^*_{u,t}-1)\kappa_{u}\pi^{e*}_t$, multiplied by the power produced $p^*_{u,t}$, i.e.,
\begin{align}
\qquad\qquad\textit{Financial Losses} = \sum_{t \in [T]}\sum_{u\in\mathcal{U}_{g}} \max&\bigg(0, (2p^*_{u,t}H_{u,2}^G + H_{u,1}^G) \psi^*_{k,t}+(2y^*_{u,t}-1)\kappa_{u}\pi^{e*}_t - \alpha_u\rho^*_{u,t}\bigg)p^*_{u,t}.\notag
\end{align}

The model runs were performed using C++/Gurobi 11.0.0 on an Apple M1, 3.2 GHz processor with 16 GB of RAM, with each run having a wall-time limit of 1 hour. The analysis covers a single time-period (i.e., $T = 1$).

\subsection{Effectiveness of the Bid-Validity Constraints}

\begin{figure}[th!]
    \centering
    \includegraphics[width=1\columnwidth]{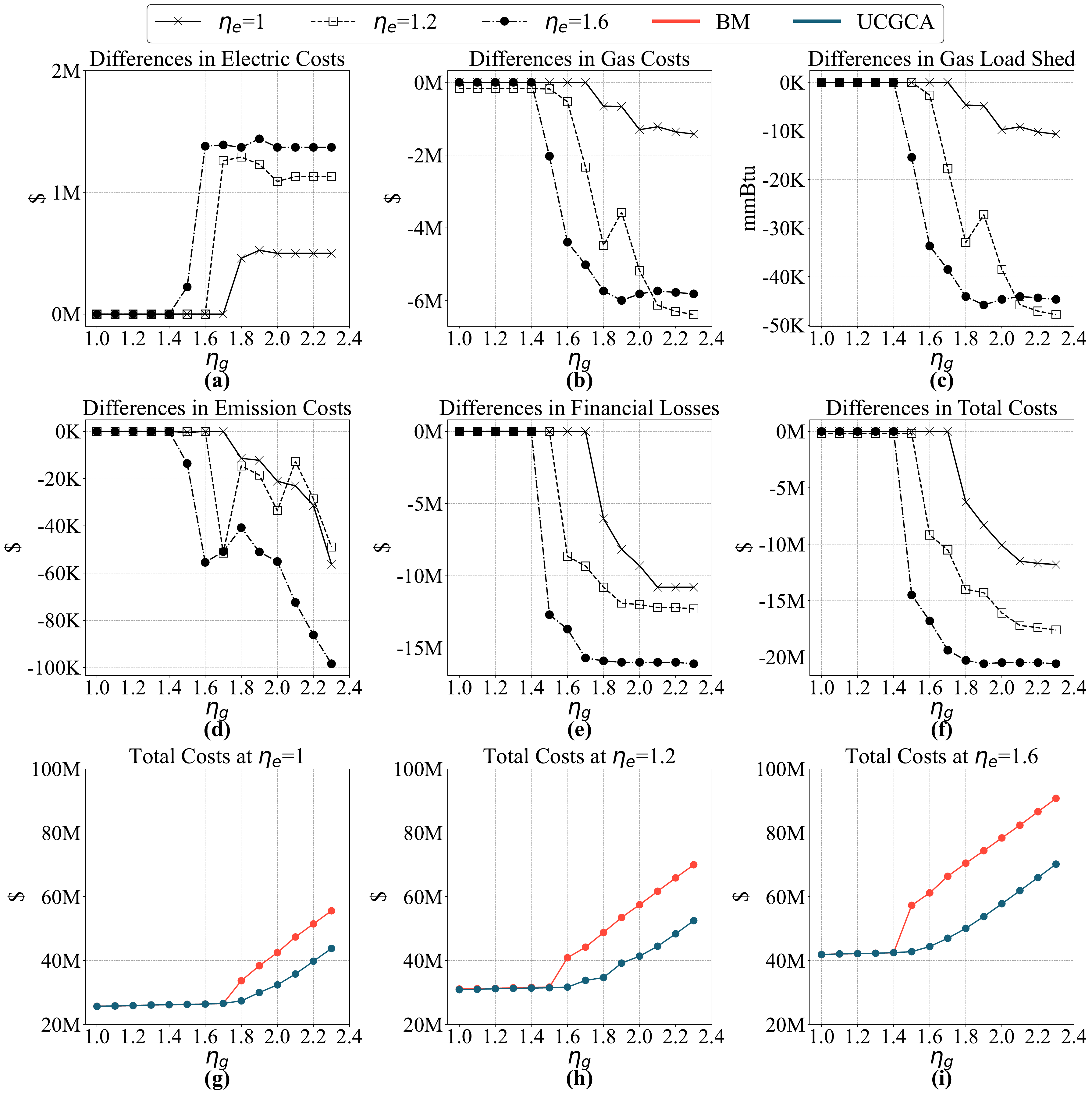}
    \caption{Cost performance comparison between BM and UCGCA under varying electrical ($\eta_e$) and gas ($\eta_g$) stress levels, with the x-axis representing $\eta_g$. Positive/negative differences indicate higher/lower UCGCA values relative to BM. Panels show differences in: (a) electric costs, (b) gas costs, (c) gas load shed, (d) emission costs, (e) financial losses, where UCGCA avoids GFPP losses through bid-validity constraints, and (f) total costs, where BM total costs include electric, gas, carbon market costs, and financial losses, while UCGCA total costs comprise only market costs because UCGCA avoids financial losses. (g)-(i) Direct comparison of total costs between models at $\eta_e = \{1, 1.2, 1.6\}$.}
    \label{fig:6}
\end{figure}
\begin{figure}[th!]
    \centering
    \includegraphics[width=1\columnwidth]{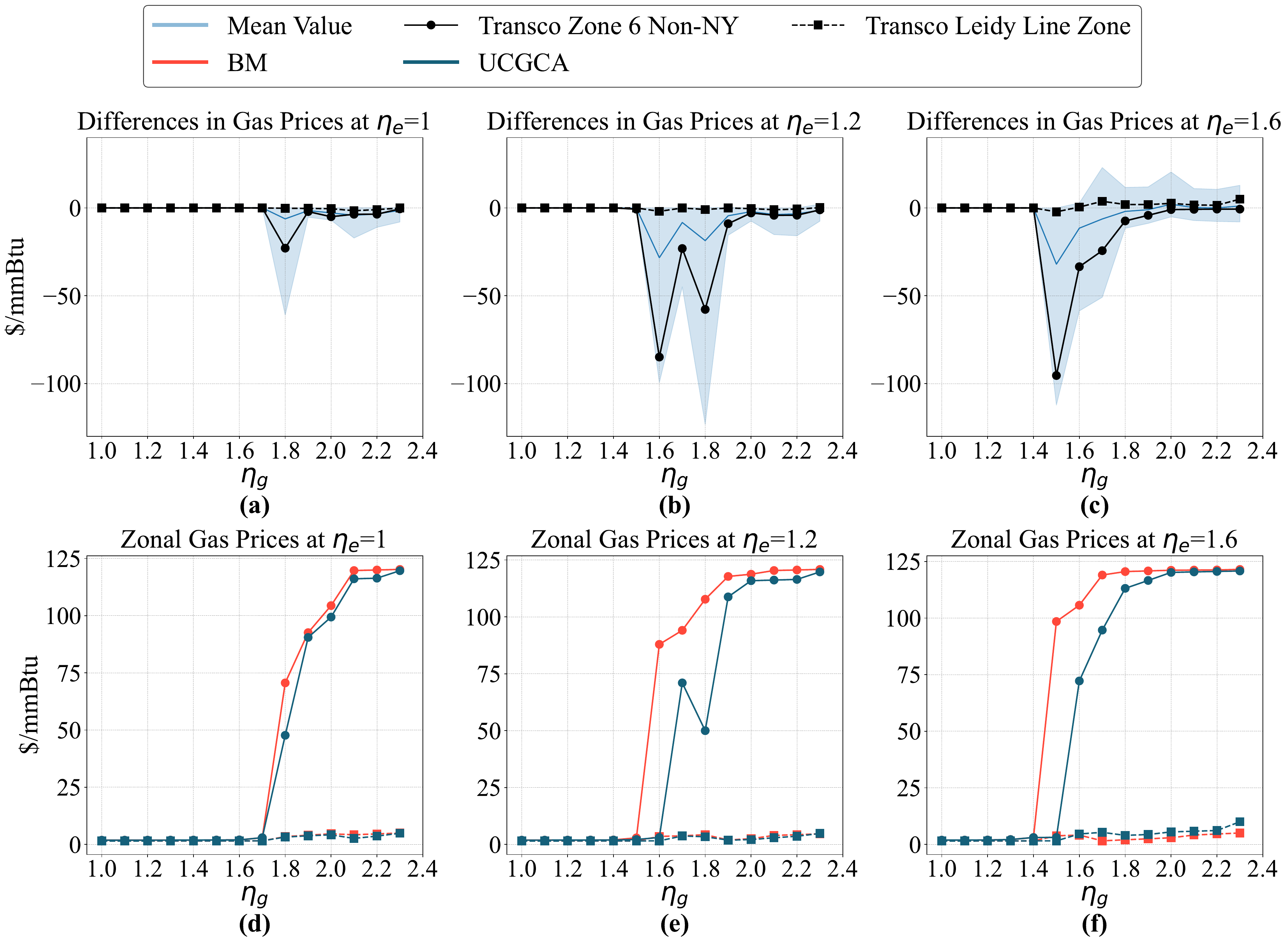}
    \caption{Comparison of Natural Gas prices between BM and UCGCA under varying electrical ($\eta_e$) and gas stresses ($\eta_g$), with the x-axis representing $\eta_g$. Note that positive/negative differences indicate higher/lower values in UCGCA compared to BM. The shaded area around the blue line indicates the 95\% percentile interval, showing the range where most of the junction prices differences fall. Plots (a)-(c) show the natural gas price differences between UCGCA and BM at electrical stress levels $\eta_e = \{1,1.2,1.6\}$, while plots (d)-(f) present the corresponding zonal gas prices for both models at the same $\eta_e$ values.}
    \label{fig:7}
\end{figure}

In this section, we compare the performance of the UCGCA and BM models under various system stress conditions. The electrical system stress is characterized by load increases of 20\% and 60\% (stress levels $\eta_e = \{1.0, 1.2, 1.6\}$, representing normal, slight, and high stress), while the gas network stress reflects load increases from 10\% to 130\% (stress levels $\eta_g = \{1.0, 1.1, \ldots, 2.3\}$). Table H.8 in the supplementary material provides additional analysis, presenting the distribution of committed generators across different fuel types for nine representative instances.

Fig.~\ref{fig:6} illustrates the differences in cost and gas load shed between the UCGCA and the BM. Fig.~\ref{fig:7} extends this comparison to differences in gas prices between the two models across all instances. Note that the positive/negative differences indicate higher/lower values achieved in the UCGCA. Fig.~\ref{fig:6} illustrates the differences in cost and gas load shed between the UCGCA and the BM. BM's total costs include electric, gas, carbon market costs, and financial losses from unprofitable gas-fired power plants. In contrast, UCGCA only incurs electric, gas, and carbon market costs, as its bid-validity constraints prevent unprofitable operations and financial losses. For direct comparison, Fig.~\ref{fig:6}(g), (h) and (i) show the actual total costs of both models. Furthermore, Fig.~\ref{fig:7} illustrates gas price differences between the models across all instances. We focus on two gas price regions (as done in \citet{8395045,8844828}), represented by black solid lines for Transco Zone 6 Non-NY and dashed lines for Transco Leidy Line Zone. These zonal prices are derived from average values of selected junctions (see Fig.~H.1 in the supplementary material). To capture system-wide effects, the blue line shows the mean price difference across \emph{all} junctions, with the shaded blue area around the blue line indicates the 95\% percentile interval, showing the range where most of the price differences occur. For direct comparison, Fig.~\ref{fig:7}(d), (e) and (f) show the actual zonal gas prices of both models.


Under normal electrical conditions ($\eta_e = 1$), the BM begins to suffer financial losses at a gas network stress level of $\eta_g = 1.8$, with financial losses escalating to \$10.8M as $\eta_g$ reaches 2.3 (see Fig.~\ref{fig:6}(e)). These losses are primarily due to unprofitable GFPPs, which result from the lack of awareness of gas and carbon market clearing outcomes. The BM also experiences higher gas load shedding due to increased network congestion, leading to significant rises in gas costs and abrupt spikes in zonal natural gas prices in the Transco Zone 6 Non-NY area. In contrast, the UCGCA model shows a more resilient response to these conditions, as evidenced by the negative differences in zonal gas prices at Transco Zone 6 Non-NY and reduced gas load shed from $\eta_g = 1.8$, as shown in Fig.~\ref{fig:7}(a) and Fig.~\ref{fig:6}(c), respectively. Additionally, the reduction in carbon allowance deficits among gas and power market participants contributes to lower overall emission costs. Although the UCGCA has increased electric costs from de-committing unprofitable GFPPs with lower bids (see the positive difference in Fig.~\ref{fig:6}(a)), its total costs are significantly reduced (see the large negative difference in Fig.~\ref{fig:6}(f)). This cost efficiency is achieved through bid-validity constraints that prevent financial losses from invalid bids, mitigate the impacts of gas congestion, and significantly lower gas costs. Fig.~\ref{fig:6}(g) presents a direct cost comparison between the two models under normal electrical system conditions ($\eta_e = 1$). The results, with BM shown in red and UCGCA in blue, demonstrate that UCGCA consistently achieves cost reductions when the gas network experiences congestion.

As electrical stress increases to $\eta_e = 1.2$, the advantages of the UCGCA’s operational adjustments become more apparent. For the BM, financial losses and gas price spikes occur earlier, starting at $\eta_g = 1.6$, with losses increasing to \$12.3M by $\eta_g = 2.3$. Conversely, the UCGCA maintains a lower cost profile even as gas stress intensifies, demonstrating effective management of network operations and costs, see the significant cost difference of \$17.6M at $\eta_g = 2.3$ in Fig.~\ref{fig:6}(f). This is due to the UCGCA’s effective selection of better-committed generators, which reduces the severity of sudden price increases in the Transco Zone 6 Non-NY area, as evidenced by the maximum price difference of \$84.88/mmBtu at $\eta_g = 1.6$ in Fig.~\ref{fig:7}(b). Note that in Fig.~\ref{fig:7}(b) and (e), there is a drop in Transco Zone 6 Non-NY gas prices at $\eta_g = 1.8$, which results from the presence of optimality gaps for some hard instances. 

The contrast between the outcomes of the BM and UCGCA becomes even more pronounced under the highest electrical stress level, $\eta_e = 1.6$. In this scenario, the UCGCA successfully avoids the significant financial losses from invalid bids that the BM begins to incur at $\eta_g = 1.5$. The UCGCA's overall costs are \$20.6M lower than those of the BM at $\eta_g = 2.3$, suggesting its superior ability to manage costs effectively. Fig.~\ref{fig:6}(i) further provides the actual values of total costs under BM and UCGCA at $\eta_e = 1.6$. Notably, at $\eta_g = 1.5$, the difference in zonal gas prices at Transco Zone 6 Non-NY between the BM and UCGCA reaches \$95.39/mmBtu (see Fig.~\ref{fig:7}(c) and (f)), indicating that the UCGCA significantly mitigates the impact of gas network congestion by committing alternative generators, thereby effectively reducing gas load shedding costs.

In summary, the UCGCA model, by incorporating anticipated outcomes from sequentially cleared markets into the unit commitment process, not only prevents financial losses for GFPPs but also moderates gas prices, reduces network congestion, and significantly lowers overall system costs.

\subsection{Effectiveness of the Dedicated Benders Decomposition Algorithm}
The computational efficiency of the UCGCA model was assessed using two approaches: conventional approach using Gurobi to solve the SLP~\eqref{eq:2} without decomposition (denoted as \texttt{C}) and a dedicated Benders Decomposition (denoted as \texttt{B}) for solving SLP~\eqref{eq:2} introduced in Section~\ref{section3b}. The dedicated Benders Decomposition is enhanced with two acceleration schemes: an in-out acceleration scheme with perturbation \citep{doi:10.1287/ijoc.2021.1128,Fischetti2016} and a normalization condition in the Benders subproblem (BSP) \citep{fischetti2010note, doi:10.1287/ijoc.2021.1128}. We evaluate the performances of \texttt{C} and \texttt{B} across various scenarios, using combinations of gas network stress ($\eta_g = \{1, 1.1, \ldots, 2.3\}$), electrical system conditions ($\eta_e = \{1, 1.2, 1.6\}$), and risk-aversion levels ($\alpha_u = \{80\%,100\%,120\%\}$). This parametric combination yields 126 distinct test instances ($14$ gas stress levels $\times$ $3$ electrical stress levels $\times$ $3$ risk factors). We adopt the conservative tolerance in Gurobi to ensure numerical stability and solution reliability. For \texttt{C}, we let Gurobi parameters with \texttt{NumericFocus=3}, \texttt{FeasibilityTol=1e-9}, \texttt{OptimalityTol=1e-9}, \texttt{IntFeasTol=1e-9} and \texttt{TimeLimit=3600}. For \texttt{B} subproblems, we let Gurobi parameters with \texttt{NumericFocus=3}, \texttt{FeasibilityTol=1e-9}, \texttt{DualReductions=0}, \texttt{BarHomogeneous=1}, \texttt{BarQCPConvTol=1e-7}, \texttt{Aggregate=0} and \texttt{ScaleFlag=0}. For \texttt{B} master problems, we let Gurobi parameters with \texttt{FeasibilityTol=1e-9}, \texttt{OptimalityTol\\=1e-9} and \texttt{IntFeasTol=1e-9}. The computational performance of \texttt{C} and \texttt{B} for varying parameters ($\eta_e$, $\eta_g$, and $\alpha_u$) is presented in Tables~H.8--H.10 of Appendix~H in the supplementary material.

The 126 test instances are categorized according to their computational complexity. We classify an instance as \textit{hard} if either \texttt{C} or \texttt{B} fails to achieve optimality within the time limit (3600~s), and as \textit{easy} if both methods converge to optimal solutions within this limit.

\begin{figure}[th]
\centering
        \includegraphics[width=0.8\columnwidth]{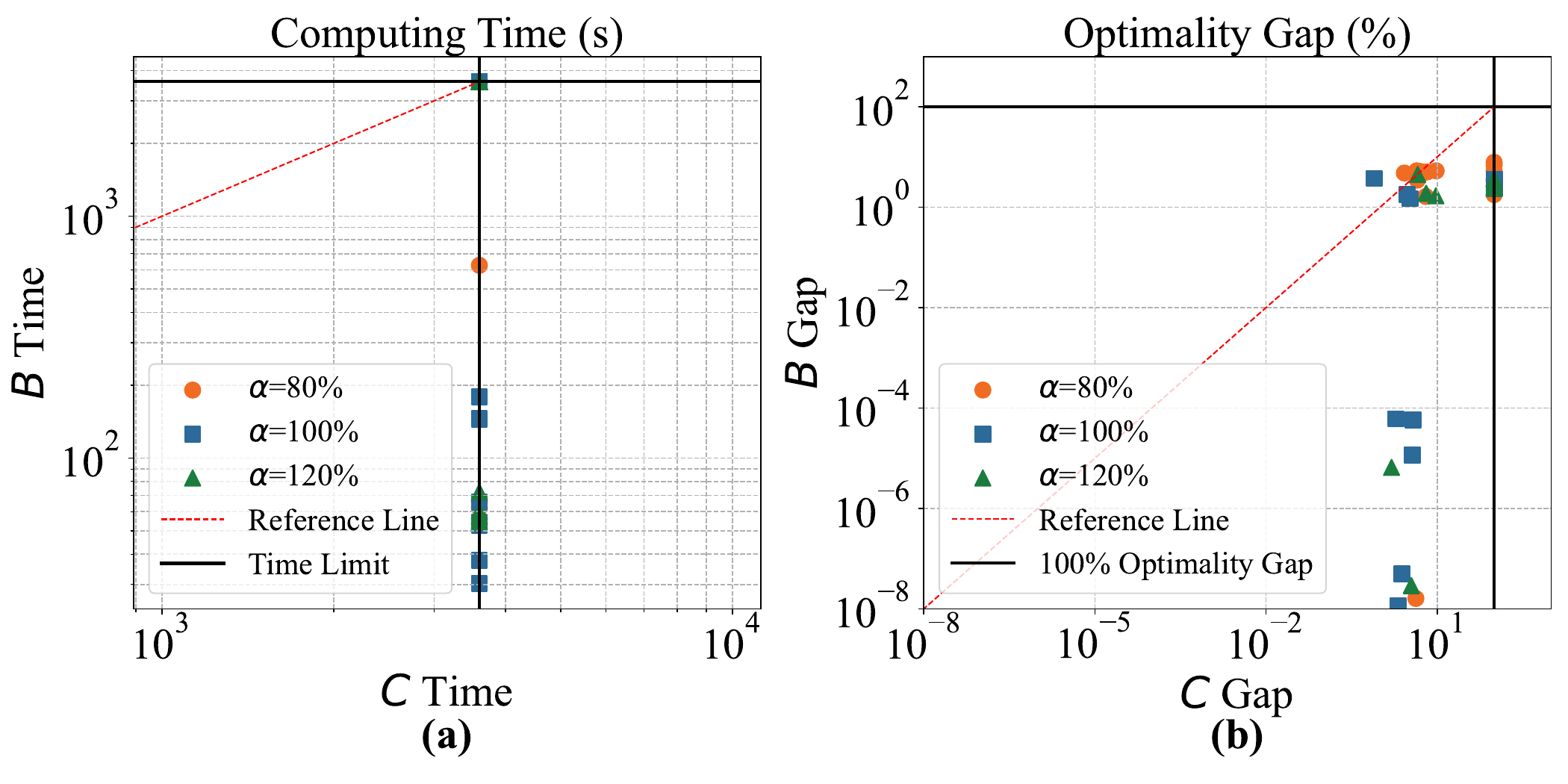}
        \caption{Comparison of computational effectiveness between using Gurobi to solve the SLP~\eqref{eq:2} without decomposition (denoted as \texttt{C}) and a dedicated Benders Decomposition (denoted as \texttt{B}) for hard instances (i.e., when the computational time of \texttt{C} or \texttt{B} reaches the 1-hour limit). (a) Computational times of \texttt{C} and \texttt{B}. (b) Optimality gaps of \texttt{C} and \texttt{B}.}
        \label{fig:computational_comparison}
    \end{figure}
Fig.~\ref{fig:computational_comparison} compares computational performance between methods \texttt{C} and \texttt{B} for hard instances. In Fig.~\ref{fig:computational_comparison}(a) and (b), the red dashed line represents equal performance between the two methods. Among 52 hard instances, \texttt{C} only outperforms \texttt{B} in only three cases ($(\alpha_u,\eta_e,\eta_g)=(80\%,1,2.1),(80\%,1,2.2),(80\%,1.2,2)$), achieving slightly smaller optimality gaps as shown by the few points above the reference line in Fig.~\ref{fig:computational_comparison}(b). However, \texttt{B} demonstrates superior performance in most hard cases, with the majority of points falling below the red dashed line, indicating both smaller optimality gaps and shorter computing times. \texttt{C}'s performance deteriorates significantly for more challenging instances, either failing to find any incumbent solution (shown by 100\% optimality gaps in Fig.~\ref{fig:computational_comparison}(b)) or producing solutions with large optimality gaps.

\begin{table}[htbp]
\centering
\caption{Averaged Computational Performance Comparison between \texttt{C} and \texttt{B}.}
\begin{tabular}{lcccc}
\hline
Instance Type & \multicolumn{2}{c}{\texttt{C}} & \multicolumn{2}{c}{\texttt{B}} \\
\cline{2-5}
& Time (s) & Gap (\%) & Time (s) & Gap (\%) \\
\hline
All Instances & 1,549.85 & 16.80 & 1,050.55 & 0.97 \\
Easy Instances & 109.20 & 0.00 & 61.45 & 0.00 \\
Hard Instances & 3,600.00 & 40.71 & 2,458.11 & 2.36 \\
\hline
\end{tabular}
\label{tab:computational_performance}
\end{table}

Table~\ref{tab:computational_performance} summarizes the computational performance of both methods across different instance types. In summary, \texttt{B} demonstrates superior computational performance compared to \texttt{C} across all test instances. On average, \texttt{B} demonstrates superior performance, reducing solution times by 32.23\% (from 1,549.85s to 1,050.55s) and optimality gaps by 94.23\% (from 16.80\% to 0.97\%). \texttt{B}'s advantage is particularly pronounced in computationally challenging instances, reducing both solution times by 31.72\% (from 3,600s to 2,458.11s) and optimality gaps by 94.20\% (from 40.71\% to 2.36\%). Even for easy instances, \texttt{B} shows efficiency improvements, solving problems 43.73\% faster than \texttt{C} (61.45s versus 109.20s). These results demonstrate that \texttt{B} not only provides high quality solutions but also achieves them more efficiently, making it particularly valuable for large-scale applications where computational performance is important.

\subsection{Strong Duality Analysis and Numerical Performance with Different Scaling Values}
\begin{figure}[th]
\centering
        \includegraphics[width=1\columnwidth]{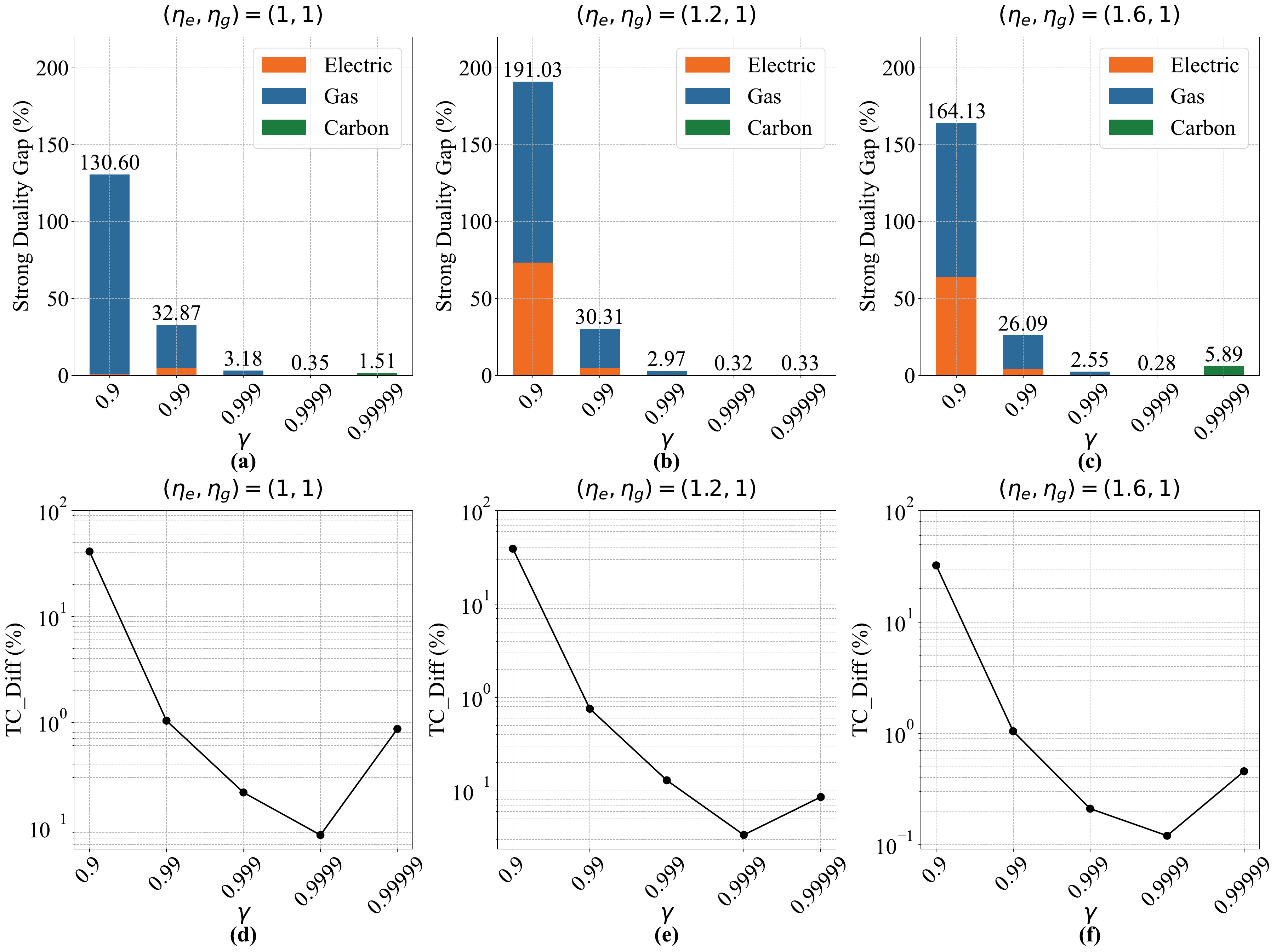}
        \caption{Strong Duality Gap (SDG) and total cost differences under different scaling factors $\gamma$ and $\eta_e$ conditions: (a)-(c) show the SDG percentages for electricity, gas, and carbon market clearing problems, and (d)-(f) present the total cost differences between UCGCA and BM for three extreme instances $(\eta_e,\eta_g) = {(1,1),(1.2,1),(1.6,1)}$.}
        \label{fig:gamma_effect}
    \end{figure}

In this section, we examine how the scaling factor $\gamma = \{0.9, 0.99, 0.999, 0.9999, 0.99999\}$ affects the strong duality condition and numerical stability for each lower-level market-clearing problem under three extreme conditions $(\eta_e,\eta_g) = \{(1,1),(1.2,1),(1.6,1)\}$. Let $(\bm{z}^*,\bm{x}^*,\bm{y}^*)$ denote the optimal solution of SLP~\eqref{eq:2}. Theorem~\ref{theorem1} proves that as $\gamma$ approaches 1, constraints \eqref{eq:2:sd} ensure strong duality for each sequential lower-level problem. We quantify the strong duality gaps (SDG) using the absolute percentage difference between primal and dual objectives relative to the primal objective for three lower-level problems:
$$
SDG_{e} = \left|\frac{c_{1}^{T}\bm{x}_{1}^* - \bm{y}_{1}^{*T}(b_1 - D_{1}\bm{z}^*)}{c_{1}^{T}\bm{x}_{1}^*}\right| \times 100\%
$$
$$
SDG_{g} = \left|\frac{c_{2}^{T}\bm{x}_{2}^* - \bm{y}_{2}^{*T}(b_2 - B_{2}\bm{x}_{1}^*- D_{2}\bm{z}^*)}{c_{2}^{T}\bm{x}_{2}^*}\right| \times 100\%
$$
$$
SDG_{c} = \left|\frac{c_{3}^{T}\bm{x}_{3}^* - \bm{y}_{3}^{*T}(b_3 - B_{3}\bm{x}_{2}^*- D_{3}\bm{z}^*)}{c_{3}^{T}\bm{x}_{3}^*}\right| \times 100\%
$$

For these three extreme cases where GFPPs incur no financial losses, the results from BM (without scaling factors) should match those from UCGCA (with scaling factors). We measure this convergence using the total cost difference between UCGCA and BM solutions, where total costs comprise electricity, gas, and carbon market costs:
$$
TC\_Diff = \left|\frac{TC^{UCGCA}-TC^{BM}}{TC^{BM}}\right| \times 100%
$$

Fig.~\ref{fig:gamma_effect} demonstrates how the strong duality gaps and total cost differences vary with $\gamma$. The upper plots (a)-(c) show that as $\gamma$ approaches 1, the total strong duality gaps decrease significantly from initial values of 130.60\%, 191.03\%, and 164.13\% (at $\gamma$ = 0.9) to less than 0.4\% (at $\gamma$ = 0.9999), aligning with Theorem~\ref{theorem1}. When $\gamma$ is small ($\gamma$ = 0.9), we observe both high SDG values and large $TC\_Diff$ values (reaching nearly 40\% as shown in plots (d)-(f)), indicating that the single-level approximation fails to approach the original multi-level problem optimal solutions. The lower plots (d)-(f) show that $TC\_Diff$ consistently exhibits a U-shaped pattern across all three $\eta_e$ cases, with values decreasing from nearly 40\% at $\gamma$ = 0.9 to a minimum of approximately 0.1\% at $\gamma$ = 0.9999. However, when $\gamma$ increases to 0.99999, both performance metrics deteriorate. Notably, at $\eta_e = 1.6$, the total SDG increases to 5.89\% (where $SDG_c$ accounts for 5.86\% of this total), and $TC\_Diff$ rises to approximately 0.4\%, indicating numerical instability at this extreme value.

Our analysis identifies $\gamma$ = 0.9999 as the optimal choice for solving UCGCA, where total SDG remains below 0.4\% for all $\eta_e$ conditions, and $TC\_Diff$ reaches its minimum value (approximately 0.1\%) for all three cases. This value achieves the best balance between approximation accuracy and computational stability across $(\eta_e,\eta_g) = \{(1,1),(1.2,1),(1.6,1)\}$.

\section{Conclusion}\label{section6}

In energy markets, sequential decision-making is fundamental to market-clearing processes, which operate across temporal, spatial, operational, and hierarchical dimensions. A critical challenge arises because unit commitment (UC) decisions must be executed in advance without knowledge of subsequent gas and carbon market outcomes. This unawareness becomes particularly problematic given the strong interdependencies between electric, natural gas and carbon markets. For example, when gas or carbon prices rise substantially, previously committed gas-fired power plants (GFPPs) in day-ahead UC may become unprofitable to operate, leading to suboptimal and economically inefficient decisions. To address these challenges, we propose a Mixed-Integer Multi-Level problem with Sequential Followers (MIMLSF) framework that explicitly models the hierarchical relationships between markets, enabling UC decisions that anticipate and account for their impacts on subsequent market conditions.

To solve this computationally challenging MIMLSF problem, we asymptotically approximate the multi-level problem as a single-level problem. Specifically, we combine lexicographic optimization with a weighted-sum method through a scaling parameter $\gamma$. This transformation preserves the sequential nature of market clearing processes while allowing the problem to be solved with commercial solvers. To enhance computational performance, we develop a dedicated Benders decomposition technique for the single-level problem. This technique first separates the complex Benders subproblem (BSP) into two tractable BSPs, and then further decomposes these two BSPs into $n$ tractable problems (where $n$ is the number of sequential problems). This multi-level BSP separation eliminates the computational burden and avoids scalability issues of solving the complex BSP as an $n$-coupled problem which involves different scaling factors.

We demonstrate our method's effectiveness by applying it to a four-level unit commitment with gas and carbon awareness (UCGCA) problem in the Northeastern United States, where gas-fired power plants (GFPPs) participate in electricity, gas, and carbon markets. This awareness is achieved by incorporating anticipated gas and carbon market outcomes into the UC problem, and those GFPPs who is expected to experience financial losses due to high gas and carbon costs will be de-committed. Compared existing approaches where system operators have no awareness from subsequent market outcomes, our approach successfully reduces total costs and mitigate gas price surges by making better unit commitment decisions. Moreover, the dedicated Benders decomposition technique achieves a 94.23\% reduction in optimality gaps and a 32.23\% reduction in computing time compared to direct solution methods. We also demonstrate that $\gamma = 0.9999$ provides the best balance between approximated solution quality and numerical stability across three extreme cases in the UCGCA problem.

Promising directions for future work include extending our deterministic MIMLSF formulation to incorporate stochastic and chance-constrained formulations, addressing the growing uncertainties with renewable energy integration in power systems. In addition, the effectiveness of multi-level Benders separation techniques will be validated through real-world case studies.


\section*{Acknowledgement}
The authors would like to thank Prof. G. Byeon and Prof. P. Van Hentenryck for their influential works \citep{8844828,doi:10.1287/ijoc.2021.1128} and the valuable data they provided.

\bibliography{ref.bib}




\end{document}